\documentclass[11pt,a4paper,oneside,reqno]{amsart} 
\input{packages_and_notation_shape.tex}

\title{Correlated Random Matrices: Band Rigidity and Edge Universality}
\author{Johannes Alt$^{\ast\dagger}$ \and L\'aszl\'o Erd\H{o}s$^{\ast\dagger}$ \and Torben Kr\"uger$^{\ddagger\mathsection}$ \and Dominik Schr\"oder$^{\ast\dagger}$}
\address{$^\ast$IST Austria, Am Campus 1, 3400 Klosterneuburg, Austria}
\address{$^\ddagger$University of Bonn, Endenicher Allee 60, 53115 Bonn, Germany}
\thanks{$^\dagger$Partially supported by ERC Advanced Grant No. 338804}
\thanks{$^\mathsection$Partially supported by the Hausdorff Center for Mathematics}
\email{dschroed@ist.ac.at} 
\email{jalt@ist.ac.at}
\email{lerdos@ist.ac.at}
\email{torben.krueger@uni-bonn.de}
\date{\today}  
\begin{document} 
\begin{abstract}
We prove edge universality for a general class of correlated real symmetric or complex Hermitian Wigner matrices with arbitrary expectation. 
 Our theorem also applies to internal edges of the self-consistent density of states. In particular,
we establish a strong form of band rigidity which excludes mismatches between location and label of eigenvalues close to internal edges in these general models. 
\end{abstract}  
\maketitle 

\section{Introduction}
Spectral statistics of large random matrices exhibit a remarkably robust universality pattern; the local distribution of eigenvalues
is independent of details of the matrix ensemble up to symmetry type. In the bulk of the spectrum this was first observed by Wigner  and formalized by Dyson and Mehta  \cite{MR0220494} who also computed  the correlation functions  of the Gaussian ensembles in the 1960's. At the spectral edges
the correct statistics was identified by Tracy and Widom both in the GUE and GOE ensembles \cite{MR1257246,MR1385083}
in the mid 1990's.  

Beyond Gaussian ensembles, the first actual proofs of universality for Wigner matrices   took different paths in the bulk and at the edge. While in the bulk 
only limited progress was made until a decade ago, the first fairly general edge universality  proof  by Soshnikov \cite{MR1727234} appeared shortly after  \cite{MR1257246,MR1385083}.
The main reason is that  edge statistics is accessible via an ingenious but laborious extension of the classical moment method of Wigner.
 In contrast, the bulk universality required fundamentally new tools based on resolvents and  the analysis of the Dyson Brownian motion developed in 
a series of work \cite{MR2964770,MR3098073,MR2662426,MR2810797,MR3372074,MR2871147}. This method, called the \emph{three-step strategy}, is summarized in \cite{MR3699468}. In certain cases
parallel results \cite{MR2669449,MR2784665} were obtained via the \emph{four moment comparison theorem}.

Despite its initial success \cite{MR1727234},
 the moment method for edge universality seems limited when it comes to generalisations beyond Wigner matrices with i.i.d.~entries; 
 the resolvent approach is much more flexible. Its primary goal is to establish \emph{local laws}, i.e.~proving that the local eigenvalue density on scales slightly above the eigenvalue spacing becomes deterministic as the dimension of the matrix tends to infinity. Refined versions of the local law even identify resolvent matrix elements with a spectral parameter very close to the real axis. In contrast to the bulk, at the spectral edge this information can be boosted to detect individual eigenvalue statistics by comparison with the Gaussian ensemble. These ideas have  led to the proof of the Tracy-Widom edge universality for Wigner matrices  with high moment conditions \cite{MR2871147}, see also   
\cite{MR2669449} with vanishing third moment. Finally, a necessary and sufficient condition on the entry distributions was found in \cite{MR3161313}
following an almost optimal necessary condition in \cite{MR2548495}. 
Direct resolvent comparison methods have  been used to prove Tracy-Widom universality for \emph{deformed Wigner matrices}, i.e.~matrices with a deterministic diagonal expectation, \cite{MR3405746}, even in a certain sparse regime \cite{MR3800840}.
The extension  of this approach to sample covariance matrices  with a  diagonal  population covariance matrix at extreme edges 
\cite{MR3582818}  has resolved a long standing conjecture in the statistics literature. Tracy-Widom universality for general population covariance matrices,
  including internal edges,  was established in~\cite{MR3704770}.

The next level of generality is to depart from the i.i.d.~case. While the resolvent method for proving local laws can handle  
\emph{generalized Wigner ensemble}, i.e.~matrices $H=(h_{ab})$ with merely stochastic variance profile $\sum_b \Var h_{ab}=1$,  
varying variances cannot be simultaneously matched with a GUE/GOE ensemble so 
the direct comparison does not work.
The  problem was resolved  in \cite{MR3253704} with a general approach that also covered invariant $\beta$-ensembles. 
While Dyson Brownian motion did not play a direct role in \cite{MR3253704}, the proof used the addition of a small Gaussian component
and  the concept of local ergodicity of the Gibbs state; ideas developed originally in \cite{MR2810797,MR2919197} in the context of bulk universality.

A fully dynamical approach to edge universality, following  an earlier development in the bulk 
based on the \emph{three-step strategy}, has recently been given in \cite{2017arXiv171203881L}. 
In general, the first step  within any three-step strategy is the
 local law providing a priori bounds. The second step is  the fast relaxation to equilibrium of the Dyson Brownian motion
that proves universality for Gaussian divisible ensembles. The third step is a perturbative  comparison argument to remove the small Gaussian component.
Recent advances in the bulk 
 have crystallized that the  only model dependent step in this strategy is the first one. The other two steps
  have been formulated as very general ``black-box'' tools whose only input is the local law
  see \cite{MR3687212,MR3729630,landon2016fixed,2017arXiv171203881L}.
  Using the three-step approach and \cite{2017arXiv171203881L},
  edge universality for sparse matrices was proved in \cite{2017arXiv171203936H} and for  
   correlated Gaussian matrices with a quite specific two-scale correlation structure  in \cite{2017arXiv171204889A}. 
 All   these edge universality results only cover the \emph{extremal edges} of the spectrum, while 
 the self-consistent (deterministic)  density of states  $\varrho$ may be supported on several intervals. 
 
 Multiple interval support becomes ubiquitous for \emph{Wigner-type} matrices \cite{MR3719056}, i.e.~matrices with independent entries and general expectation and variance profile. A prerequisite for Tracy-Widom universality, the  square root singularity in the density, even at the \emph{internal edges}, is a universal phenomenon for a very large class of random matrices since it is  inherent to the underlying \emph{Dyson equation}. This was demonstrated  for  Wigner-type matrices  in \cite{2015arXiv150605095A} and 
  here we extend it for correlated random matrices with a general correlation structure. We remark that
 a second singularity type, the \emph{cubic root cusp},  is also possible; the corresponding analysis of the Dyson equation is given 
 in \cite{shapepaper}, while the optimal local law and the universal spectral statistics are proven in \cite{1809.03971,1811.04055}. 

In the current paper we show that the eigenvalue statistics at the spectral edges of  
 $\varrho$
 follow the Tracy-Widom distribution, assuming only a mild decay of correlation between entries, but otherwise no special structure. 
 We can handle any internal edge as well.  
In the literature internal edge universality for matrices of Wigner-type has first been established for deformed GUE ensembles \cite{MR2799948} 
which critically relied on contour integral methods, only available for Gaussian models in the Hermitian symmetry class. 
A similar method handled extreme eigenvalues of deformed GUE \cite{MR2288065,MR3500269}. 
A more general  approach for internal edges has been given in \cite{MR3704770} that could handle any deformed Wigner  matrices with general expectation,
as long as the variance profile is constant, by comparing it with the corresponding Gaussian model.  Our method requires neither constant variance nor independence
of the matrix elements. 

 The proof of
our general form of edge universality at all internal edges follows the three-step strategy and uses  the recent paper \cite{2017arXiv171203881L}
for the second step and well established  canonical arguments for the third step that will be summarized.
 The backbone of the work is thus the first step, an optimal local law at the spectral edges, the proof of which has two well separated components; a probabilistic and a deterministic one. The probabilistic component 
  is 
   insenstive to the location in the spectrum and follows directly from \cite{2017arXiv170510661E}. 
  Here we present a compact and practically self-contained proof of the deterministic component of the local law that can be followed without consulting previous works;
we only rely on some general results from  functional analysis  proven in \cite{2016arXiv160408188A} and some minor technicality
on 
the Dyson equation 
from \cite{shapepaper}.
First, we develop
a detailed shape analysis of the self-consistent  density  $\varrho$ near the regular edges, generalizing the previous  bulk result from 
\cite{2016arXiv160408188A} and the singularity analysis in the independent 
case from \cite{2015arXiv150605095A}.
    Second, we prove a strong version of the local
law that excludes eigenvalues in the internal gaps. Third, we establish a  topological rigidity phenomenon for the \emph{bands}, the connected components that constitute the support of  $\varrho$. 

 \emph{Band rigidity}   is a new phenomenon for the Dyson equation and it 
asserts that the number of eigenvalues within each band exactly matches the mass that $\varrho$ predicts for that band. The topological nature of band rigidity guarantees that this mass  remains constant along the deformations of the model
as long as the gaps between the bands remain open. A similar rigidity (also called ``exact separation of eigenvalues'')
has first been established for sample covariance matrices in \cite{MR1733159}
and it also played a key role in Tracy-Widom universality proof at internal edges in \cite{MR3704770}. 
  Note that band rigidity is a much stronger concept than the customary rigidity in random matrix theory \cite{MR2871147} that allows for an uncertainty in the location of $N^\epsilon$ eigenvalues. 
  In other words, there is no mismatch whatsoever
between location and label of the eigenvalues near the internal edges along the matrix  Dyson Brownian motion, the label of the eigenvalue uniquely determines to which spectral band it belongs.

Our result 
 highlights a key difference between Wigner-type matrix models and invariant $\beta$-ensembles. For self-consistent densities with multiple support intervals (the so called \emph{multi-cut} regime), the number of particles (eigenvalues) close to some support interval fluctuates for invariant ensembles with general potentials \cite{2013arXiv1303.1045B}. As a consequence internal edge universality results (see e.g.~\cite{MR2012268,MR3756421}) require a stochastic relabelling of eigenvalues.

Our setup is a general $N\times N$ random  matrix $H=H^*$ with a slowly decaying correlation structure and arbitrary expectation, under the very same 
general conditions as the recent bulk universality result from \cite{2017arXiv170510661E}. 
The starting point is to find the deterministic approximation of the resolvent $G(z)=(H-z)^{-1}$ with a complex spectral parameter $z$ in the upper half plane. This approximation is given by the solution $M=M(z)$ to the \emph{Matrix Dyson Equation (MDE)}, see \eqref{MDE} below. The resolvent $G(z)$ approximately  satisfies  the MDE
with an additive perturbation term  which was already shown to be sufficiently small in \cite{2017arXiv170510661E}. This fact,  combined with a careful stability and shape analysis 
 of the MDE in Section~\ref{sec:dens_close_to_edge}
  imply that $G$ is indeed close to $M$. In order to prove edge universality we use a correlated Ornstein-Uhlenbeck process $H_t$ which adds a small Gaussian component of size $t$ to the original matrix model, while preserving expectation and covariance. We prove that the resolvent satisfies the optimal local law uniformly along the flow and appeal to the recent result from \cite{2017arXiv171203881L} to prove edge universality for $H_t$ whenever $t\gg N^{-1/3}$. In the final step we perform a resolvent 
  comparison together with our band rigidity
 to show that the eigenvalue correlation functions of $H_t$ matches those of $H$ as long as $t\ll N^{-1/6}$ which yields the desired edge universality.

After presenting our main results in Section \ref{section main result}, we then prove the optimal local law 
in Section \ref{sec proof local law}. Section~\ref{sec:dens_close_to_edge} contains the analysis of the MDE. Both types of rigidity are shown in Section \ref{sec:band_rigidity}. Section \ref{sec proof edge univ} is devoted to the proof of edge universality. 

\subsection*{Notations}
If for some constants $c,C>0$ it holds that $f\le C g$ or $c g\le f\le C g$, then we write $f\lesssim g$ and $f\sim g$, respectively. These constants $c,C$ may depend on some basic parameters which we call model parameters later. We denote vectors by bold-faced lower case Roman letters $\vx,\vy\in\C^N$, and matrices by upper case Roman letters $A,B\in\C^{N\times N}$. The standard scalar product and Euclidean norm on $\C^N$ will be written as $\braket{\vx,\vy}$ and $\norm{\vx}$, while we also write $\braket{A,B}\defeq N^{-1}\Tr A^\ast B$ for the scalar product of matrices, and $\braket{A}\defeq N^{-1}\Tr A$. The usual operator norm induced by the vector norm $\norm{\cdot}$ will be denoted by $\norm{A}$, while the Hilbert-Schmidt (or Frobenius) norm will be denoted by $\norm{A}_\text{hs}\defeq \sqrt{\braket{A,A}}$. 
The operator norms induced on linear maps $\C^{N\times N} \to \C^{N\times N}$ 
by $\normtwo{\genarg}$ and $\norm{\genarg}$ are denoted by $\normsp{\genarg}$ and $\norm{\genarg}$, respectively. The identity matrix in $\C^{N\times N}$ is indicated by $\id$ and the 
identity mapping on $\C^{N\times N}$ by $\Id$.  
For random variables $X,Y,\dots$ we denote the joint cumulant by $\kappa(X,Y,\dots)$. For integers $n$ we define $[n]\defeq\{1,\dots,n\}$.

\section{Main results}\label{section main result}
We consider correlated real symmetric and complex Hermitian random matrices of the form 
\begin{equation*} 
H=A+W, \qquad \E W=0 
\end{equation*}
with deterministic $A\in\C^{N\times N}$ and sufficiently fast decaying correlations among the matrix elements of $W$. The matrix entries $w_{ab}=w_\alpha$ are often labelled by double indices $\alpha=(a,b)\in[N]^2$. The randomness $W$ is scaled in such a way that $\sqrt N w_\alpha$ are random variables of order one\footnote{In some previous works, as in \cite{2017arXiv170510661E}, the convention $H=A+W/\sqrt N$ with order one $w_\alpha$ was used.}. This requirement ensures that the  size of the  spectrum of $H$ is kept of order $1$, as $N$ tends to infinity. Our first aim is to prove that the resolvent $G=G(z)=(H-z)^{-1}$ is well approximated by the solution $M=M(z)$ to the \emph{Matrix Dyson equation (MDE)}
\begin{equation}\label{MDE}
\id+(z-A+\SS[M])M=0,\quad \Im M\defeq \frac{M-M^\ast}{2\ii}>0,\quad \SS[R]\defeq \E W R W, \quad z\in\HC\defeq\Set{z\in\C|\Im z>0}
\end{equation}
in a neighbourhood around the edges of the spectrum. We suppress the dependence of $G$ and $M$, and similarly of many other quantities, on the spectral parameter $z$ in our notation. Estimates on $z$-dependent quantities are always meant uniformly for $z$ in some specified domain. From the solution $M$ we define $\varrho \colon \Hb \to \R$ and extend it to the real line 
\begin{equation}\label{definition scDOS}
\varrho(z) \defeq \frac{1}{\pi} \Im \braket{M(z)}, \quad z \in \Hb, \qquad  
\varrho(\tau)\defeq \lim_{\eta\searrow 0}\dens(\tau + \ii \eta), \quad  \tau\in\R. 
\end{equation}
 By \cite[Proposition 2.2]{2016arXiv160408188A} the limit in \eqref{definition scDOS} exists and $\varrho$ is a H\"older continuous function on $\Hb \cup \R$ under Assumptions \ref{assumption A} and \ref{assumption flatness} below. 
The \emph{self-consistent density of states} is the restriction of $\varrho$ to $\R$
which approximates the density of eigenvalues of $H$ increasingly well as $N$ tends to infinity. 
Its support, $\supp\varrho \subset \R$, is called the \emph{self-consistent spectrum}.
We remark that $\varrho$ on $\Hb$ is the harmonic extension of $\varrho|_\R$. 
We now list our main assumptions, which are identical to those from \cite{2017arXiv170510661E}, apart from  the additional 
Assumption~\ref{assumption bounded M}, which was automatically satisfied in \cite{2017arXiv170510661E}, i.e.~in the bulk regime (cf.~Remark~\ref{rmk:boundedness of M} below).  
All constants in Assumptions \ref{assumption A}--\ref{assumption bounded M} and Definition \ref{def regular edge} are called \emph{model parameters}. 

\begin{assumption}[Bounded expectation]\label{assumption A}
There exists some constant $C$ such that $\norm{A}\le C$  for all $N$.
\end{assumption}
\begin{assumption}[Finite moments]\label{assumption high moments}
For all $q\in\N$ there exists a constant $\mu_q$ such that $\E \abs[0]{\sqrt N w_{\alpha}}^q\le \mu_q$ for all $\alpha.$
\end{assumption}
\stepcounter{assumption}\stepcounter{assumption}
\begin{assCD}[\hypertarget{assumpCD}Polynomially decaying metric correlation structure]
For the $k=2$ point correlation we assume
\begin{subequations}\label{metric tree decay CD}
\begin{align}\abs{\kappa\Big(f_1(\sqrt N W),f_2(\sqrt N W)\Big)} &\le C_2 \frac{\sqrt{\E \abs[1]{f_1(\sqrt N W)}^2}\sqrt{\E \abs[1]{f_2(\sqrt N W)}^2}}{1+d(\supp f_1,\supp f_2)^s} , 
\label{cor decay}
\intertext{for some $s>12$ and all square integrable functions $f_1,f_2$. For $k\ge 3$ we assume a decay condition of the form}
\label{tree decay}\abs{\kappa\Big(f_1(\sqrt N W),\dots,f_k(\sqrt NW)\Big)} &\le C_k \prod_{e\in E(T_{\text{min}})} \abs{\kappa(e)}, \end{align}
\end{subequations}
where $T_{\text{min}}$ is the minimal spanning tree in the complete graph on the vertices $1,\dots,k$ with respect to the edge length $\dist(\{i,j\})=d(\supp f_i,\supp f_j)$, i.e.~the tree for which the sum of the lengths $\dist(e)$ is minimal, 
and $\kappa(\{i,j\})=\kappa(f_i,f_j)$. Here $d$ is the standard Euclidean metric on the index space $[N]^2$  and $\supp f \subset [N]^2$ denotes  the set indexing all entries in $\sqrt N W$ that $f$ genuinely depends on, and $C_k<\infty$ are some absolute constants.
\end{assCD}
\begin{remark}
All results in this paper and their proofs hold verbatim  if Assumption \hyperlink{assumpCD}{(CD)} is replaced by the more general assumptions (C),(D) from \cite{2017arXiv170510661E}. In particular, the metric structure imposed on the index space $[N]^2$ is not essential. For details the reader is referred to \cite[Section 2.1]{2017arXiv170510661E}.
\end{remark}
\begin{assumption}[Flatness]\label{assumption flatness} 
There exist constants $0<c<C$ such that $c\braket{T} \le \SS[T] \le C \braket{T}$ for any positive semi-definite matrix $T$.    
\end{assumption}
\begin{assumption}[Fullness]\label{assumption fullness}
There exists a constant $\lambda>0$ such that $N\E \abs{\Tr B W}^2 \ge \lambda \Tr B^2$ for any deterministic matrix $B$ of the same symmetry class  (either real symmetric or complex Hermitian)  as $H$.
\end{assumption}

\begin{assumption}[Bounded self-consistent Green function]\label{assumption bounded M}
There exist constants $\omega_\ast, M_\ast>0$ such that 
\[
\sup_z \norm{M(z)} \le M_\ast,
\] 
where the supremum is taken over all $z \in \mathbb{H}$ with $\abs{\Re z-\tau_0}\le \omega_\ast$ and $0 < \Im z \leq 1$.
\end{assumption}
\begin{remark}
Assumption~\ref{assumption flatness} is an effective mean field condition that provides upper and lower bounds on the variances of the entries of $W$. In fact it is equivalent to $\E\abs{\braket{\vx,W\vy}}^2 \sim 1/N$ for all normalised $\vx,\vy \in \C^N$.
Assumption~\ref{assumption fullness} is equivalent to $\SS-\lambda \SS_{\mathrm{G}}$ remaining positivity preserving, where  $\SS_{\mathrm{G}}$ is the self-energy operator of a full GUE/GOE matrix.
\end{remark}
\begin{remark} \label{rmk:boundedness of M}
The boundedness of $\norm{M}$ is automatically satisfied in the spectral bulk. At the edges, however, the boundedness cannot be guaranteed under Assumptions \ref{assumption A}--\ref{assumption flatness} but has to be  verified for each concrete model (see \cite[Section 9]{shapepaper} for a large class of models for which $\norm M$ is guaranteed to be bounded).
\end{remark}

Our main technical result is an optimal local law at \emph{regular edges} $\tau_0 \in \partial \supp \varrho$ asserting that $G(z)=(H-z)^{-1}$ is well approximated by $M(z)$ in the $N \to \infty$ limit. Around such an edge we consider the domain of spectral parameters $z=\tau+\ii \eta$ whose imaginary part $\Im z =\eta$ is slightly larger than $1/N$, i.e.~in the spectral domain
\begin{equation}\label{spectral domains}
\DD^\delta_\gamma\defeq \Set{z\in\DD^\delta|\Im z\ge N^{-1+\gamma} } \qquad \text{with}\qquad \DD^\delta\defeq \Set{\tau+\ii \eta  | \abs{\tau-\tau_0}\le \delta\,, 0<\eta\le 1}
\end{equation}
for  any $\gamma,\delta>0$. 
\begin{definition}[Regular edge]\label{def regular edge}
We call an edge $\tau_0\in\partial\supp\varrho$ \emph{regular} if the limit
\begin{equation}\label{sqrt growth}
\lim_{\supp\varrho\ni\tau\to\tau_0} \frac{\varrho(\tau)}{\sqrt{\abs{\tau-\tau_0}}} = \frac{\gamma_\mathrm{edge}^{3/2}}{\pi}
\end{equation}
exists for some \emph{slope parameter} $\gamma_\mathrm{edge}$ that satisfies $0<c_\ast\le \gamma_\mathrm{edge}\le c^\ast<\infty$ for some constants $c_\ast,c^\ast$.
\end{definition}
\begin{remark}\label{rem:regular_edge_main_results}
We remark that there are several equivalent characterisations of \emph{regular edges}. We chose \eqref{sqrt growth} here because it highlights that the essential prerequisite for Tracy-Widom 
universality is a local square-root singularity. According to the classification result from \cite{shapepaper} it follows that \eqref{sqrt growth} is equivalent\footnote{
In fact, in \cite[Section 7.6]{shapepaper} it is proven that if the self-consistent spectrum $\supp\dens$ has a macroscopic gap next to some $\tau_0 \in \pt\supp\dens$,
then $\dens$ has a square root behaviour at $\tau_0$. Together with Theorem~\ref{thm:dens_close_to_regular_edge} later, this shows that regular edges in the sense of \eqref{sqrt growth} are precisely 
those $\tau_0 \in \pt\supp\dens$ which are adjacent to macroscopic gaps.} to assuming that the gap in $\supp\varrho$ adjacent to $\tau_0$ is of size $\gtrsim 1$. 
\end{remark}	

\begin{theorem}[Edge local law]\label{theorem local law} Let Assumptions \ref{assumption A}--\ref{assumption flatness} and \ref{assumption bounded M} be satisfied for some regular edge $\tau_0 \in \partial \supp \varrho$. Then for any $D,\gamma,\epsilon>0$ and sufficiently small $\delta>0$, there exists some $C<\infty$ depending only on these and the model parameters such that with $G=G(z)$ and $M=M(z)$ we have the isotropic local law,
\begin{subequations}
\begin{equation}\label{iso bulk}
\P \left(\abs{\braket{\vx,(G-M)\vy}}\le N^{\epsilon}\norm{\vx}\norm{\vy}\left(\sqrt{\frac{\varrho}{N\Im z}}+\frac{1}{N\Im z}\right)\quad \text{in}\quad\DD^\delta_\gamma\right) \ge 1 - C N^{-D}
\end{equation}
for all deterministic vectors $\vx,\vy\in \C^N$ and the averaged local law,
\begin{equation}\label{av bulk} 
\P \left(\abs{\braket{B(G-M)}}\le N^{\epsilon}\frac{\norm{B}}{N\Im z}\quad \text{in}\quad\DD^\delta_\gamma\right) \ge 1 - C N^{-D}
\end{equation}
for all deterministic matrices $B\in\C^{N\times N}$. 
Moreover, at a distance at least $N^{-2/3+\epsilon}$ away from the self-consistent spectrum we have the improved averaged local law for any $\epsilon>0$
\begin{equation}\label{av outisde}
\P \left(\abs{\braket{B(G-M)}}\le \frac{N^{\epsilon}\norm{B}}{N\dist( z,\supp\varrho)}\,\,\,\text{in}\,\,\,\Set{z\in\DD^\delta|\frac{\dist( z,\supp\varrho)}{N^{-2/3+\epsilon}}\ge1}\right) \ge 1 - C N^{-D}
\end{equation}
with $C$ also depending on $\epsilon$.
\end{subequations}
\end{theorem}

\begin{corollary}[No eigenvalues outside the support of the self-consistent density]\label{cor no eigenvalues outside}
Under the assumptions of Theorem~\ref{theorem local law} we have for any $\epsilon,D>0$ and sufficiently small $\delta>0$
\[\P\left(\exists\lambda\in\Spec H \,\big| \,\abs{\tau_0-\lambda}\le \delta\,,\;\dist(\lambda,\supp\varrho)\ge N^{-2/3+\epsilon}\right)\le_{\epsilon,D} N^{-D},\]
where $\le_{\epsilon,D}$ means a bound up to some multiplicative constant $C=C(\epsilon,D)$.
\end{corollary}
\begin{corollary}[Delocalisation]\label{cor delocalisation}  
Under the assumptions of Theorem \ref{theorem local law} it holds for an $\ell^2$-normalized eigenvector $\bm u$ corresponding to an eigenvalue $\lambda$ of $H$ close to the edge $\tau_0$ that
\[ \sup_{\norm{\vx} = 1} \P\left(\abs{\braket{\vx,\vu}} \ge \frac{N^\epsilon}{\sqrt N}\,\big|\, H\vu=\lambda\vu,\,\norm{\vu}=1,\, \abs{\tau_0-\lambda}\le \delta\right)\le_{\epsilon,D} N^{-D}\]
for any $\epsilon,D>0$ and sufficiently small $\delta >0$.
\end{corollary}
\begin{corollary}[Band rigidity and eigenvalue rigidity]\label{cor rigidity}
Under the assumptions of Theorem \ref{theorem local law} the following holds. For any $\epsilon,D>0$ there exists some $C<\infty$ such that for any $\tau\in\R\setminus\supp\varrho$ with $\dist(\tau,\supp\varrho)\ge\epsilon$ the number of eigenvalues less than $\tau$ is with high probability deterministic, i.e.~that
\begin{subequations}
\begin{equation}\label{det num of evs}
\P\bigg(\abs{\Spec H\cap (-\infty,\tau)} = N \int_{-\infty}^\tau\varrho(x)\diff x\bigg)\ge 1- CN^{-D}.
\end{equation}
We also have the following strong form of eigenvalue rigidity in a neighbourhood of a regular edge $\tau_0$. Let $\lambda_1\le \dots\le\lambda_N$ be the ordered eigenvalues of $H$ and denote the index of the $N$-quantile close to energy $\tau\in \interior(\supp \varrho)$ by $k(\tau)\defeq \lceil N\int_{-\infty}^\tau \varrho(x)\diff x\rceil$. It then holds that
\begin{equation}\label{eq rigidity}
\P\left( \sup_{\tau} \abs{\lambda_{k(\tau)} - \tau } \ge\min \bigg\{ \frac{N^\epsilon}{N\abs{\tau-\tau_0}^{1/2}}, \frac{N^\epsilon}{N^{2/3}} \bigg\}\right)\le_{\epsilon,D} N^{-D}
\end{equation}
\end{subequations}
 for any $\epsilon,D>0$ and sufficiently small $\delta>0$, where the supremum is taken over all $\tau\in\supp\varrho$ such that  $\abs{\tau-\tau_0}\le \delta$. 
\end{corollary}
\begin{remark}[Integer mass]\label{remark integer mass}
Note that \eqref{det num of evs} entails the non trivial fact that for $\tau\not\in\supp\varrho$, $N\int_{-\infty}^\tau \varrho(x) \diff x$ is always an integer, see Proposition~\ref{prp:Band mass formula} below.  Moreover, it then trivially implies that $N\int_a^b\varrho(x)\diff x$ is an integer for each \emph{spectral band} $[a,b]$, i.e.~connected component of $\supp\varrho$. Finally, \eqref{det num of evs} also shows that the number of eigenvalues in each band is given by this integer with overwhelming probability. This is in sharp contrast to invariant $\beta$-ensembles where no such mechanism is present. For example, for an odd number of particles in a symmetric double-well potential, $N\int_{-\infty}^0\varrho(x)\diff x=N/2$ is a half integer.
\end{remark}

The main application of the optimal local law from Theorem \ref{theorem local law} is edge universality, as stated in the following theorem, generalising several previous edge universality results listed in the introduction. For definiteness we only state and prove the result for regular right edges. The corresponding statement for left edges can be proven along the same lines.
\begin{theorem}[Edge universality]\label{thm edge univ}
Under the Assumptions \ref{assumption A}--\ref{assumption bounded M} the following statement holds true. Assume that $\tau_0\in \partial \supp \varrho$ is a right regular edge of $\varrho$ with \emph{slope parameter} $\gamma_\mathrm{edge}$ as in Definition \ref{def regular edge}. The integer (see Remark \ref{remark integer mass}) $i_0\defeq N\int_{-\infty}^{\tau_0}\varrho(x)\diff x$ labels the largest eigenvalue $\lambda_{i_0}$ close to the band  edge $\tau_0$ with high probability. Furthermore, for test functions $F\colon\R^{k+1}\to\R$ such that $\norm{F}_\infty+\norm{\nabla F}_\infty\le C<\infty$ we have 
\begin{equation*}
\abs{\E F\Big(\gamma_\mathrm{edge} N^{2/3}(\lambda_{i_0}-\tau_0),\dots,\gamma_\mathrm{edge} N^{2/3}(\lambda_{i_0-k}-\tau_0)\Big)- \E F\Big( N^{2/3}(\mu_N-2),\dots, N^{2/3}(\mu_{N-k}-2)\Big) }\lesssim N^{-c}
\end{equation*}
for some $c=c_k>0$. Here $\mu_1,\dots,\mu_N$ are the eigenvalues of a standard GUE/GOE matrix, depending on the symmetry class of $H$.
\end{theorem}
From Theorem \ref{thm edge univ} we can immediately conclude that the eigenvalues of $H$ near the regular edges follow the Tracy-Widom distribution. We remark that the direct analogue of Theorem \ref{thm edge univ} does not hold true for invariant $\beta$-ensembles with a \emph{multi-cut} density. This is due to the fact that the number of particles close to a band of the self-consistent density, commonly known as the \emph{filling fraction}, is known to be a fluctuating quantity for general classes of potentials. We refer the reader to \cite{MR1790279} for a description of this phenomenon, to \cite{MR2268864,MR3063494} for non-Gaussian linear statistics in the multi-cut regime and to \cite{2013arXiv1303.1045B} for results on the fluctuations of filling fractions. Variants of Theorem \ref{thm edge univ} which allow for a relabelling of eigenvalues for invariant $\beta$-ensembles can be found in \cite{MR2012268,MR3756421}.

\section{Proof of the local law}\label{sec proof local law}
The proof of a local law consists of three largely separate arguments. The first part concerns the analysis of the \emph{stability operator}
\begin{equation} \label{eq:def_BO} 
\BO[R] \defeq R - M\SS[R]M
\end{equation} 
for $R \in \C^{N\times N}$,  
 and shape analysis of the solution $M$ to \eqref{MDE}. The second part is proving that the resolvent $G$ is indeed an approximate solution to \eqref{MDE} in the sense that 
\begin{equation}\label{D def}
D\defeq \id+(z-A+\SS[G])G = WG+\SS[G]G
\end{equation} 
is small. Finally, the third part consists of a bootstrap argument starting in the domain $\DD_1^\delta$ and iteratively increasing the domain to $\DD_\gamma^\delta$ while maintaining the desired bound on $G-M$. 

\subsection{Stability}\label{sec stability}
From \eqref{MDE} and \eqref{D def}, we see that the difference between $G$ and $M$ is described by the relation 
\begin{equation}\label{G-M D eq}
\BO[G-M]=-MD+M\SS[G-M](G-M).
\end{equation}
 To prove estimates on $G-M$ we need to analyse $\BO$, the stability operator. Near the edge we will demonstrate
that  $\BO$ has a very small (in absolute value) simple eigenvalue, that we will denote by $\beta$, and it turns out that
$\beta$ is well separated away from the rest of the spectrum of $\BO$. Let $P$ and $B$ denote the  corresponding
left and right eigenvectors of $\BO$,  i.e.~$\BO^*[P] = \bar \beta P$ and $\BO[B] = \beta B$,  and we will specify their normalisation later.
 Note that $\BO$ is typically not self-adjoint, so $P\ne B$. Since $\beta$ is small, $\BO^{-1}$ 
 is unstable in the direction of the eigenspace of $\beta$.
 We therefore   separate this unstable direction by  writing $G-M = \Theta B + \mbox{\emph{Error}}$ where 
 \begin{equation}\label{def:Theta}
 \Theta\defeq \frac{\braket{P,G-M}}{\braket{P,B}}
 \end{equation}
  is the key quantity  and the error term lies
 in spectral subspace complementary to $B$. We will then 
   establish bounds in terms of $\Theta$ 
and $D$ from \eqref{G-M D eq}. 
We note that this separation is not necessary in the bulk regime studied in \cite{2017arXiv170510661E},  where the stability operator is bounded 
in every direction,  which explains the additional complexity of the proof of 
Theorem~\ref{theorem local law} compared to the bulk local law in \cite{2017arXiv170510661E}.

The reader should not be confused by the term ``eigenvector'' in the context of operators $\C^{N\times N} \to 
\C^{N \times N}$ as eigenvectors are in fact matrices in this setting, e.g.~the eigenvectors $P$ and $B$ 
of $\BO$ above are actually matrices in $\C^{N\times N}$.

We begin by collecting some qualitative and quantitative information about the MDE and its stability operator,  which will be proven in Section~\ref{subsec:proof_of_pro_stab_mde} below. 
We note that \eqref{prop31 unique} was first obtained in \cite{MR2376207} and \eqref{prop31 dens} goes back to \cite{2016arXiv160408188A}. 
\begin{proposition}[Stability of MDE and properties of the solution]\label{prop stability reference} The following hold true under Assumption \ref{assumption A}, \ref{assumption flatness} 
and \ref{assumption bounded M} for some $\tau_0 \in \R$.  
\begin{enumerate}[(i)]
\item\label{prop31 unique} The MDE \eqref{MDE} has a unique solution $M=M(z)$ for all $z\in\HC$ and moreover the map $z\mapsto M(z)$ is holomorphic.
\item\label{prop31 dens} \label{mu def} The holomorphic function $\braket{M}\colon \HC\to\HC$ is the Stieltjes transform of a compactly supported probability measure with continuous density $\varrho\colon\R\to[0,\infty)$ given by \eqref{definition scDOS}. Moreover, $\varrho$ is real analytic on the open set $\set{\varrho>0}$.
\end{enumerate}
If $\tau_0\in\pt\supp\dens$ is a regular edge then there is $\delta_* \sim 1$ such that, 
for all $z \in\Hb$ satisfying $\abs{z- \tau_0} \leq \delta_*$, we have 
\begin{enumerate}[(i)]
\addtocounter{enumi}{2}
\item \label{prop sqrt edge} The harmonic extension of the self-consistent density of states 
scales like  
 \[ \dens(z) \sim \begin{cases} \sqrt{\kappa + \eta}, & \text{if } \tau \in \supp\dens, \\  \eta/\sqrt{\kappa + \eta}, & \text{if } \tau \notin \supp\dens, \end{cases}\]
where $\tau = \Re z$, $\eta = \Im z$ and $\kappa\defeq\abs{\tau-\tau_0}$. 
\item\label{1-CMS bound with flatness} 
There exist $P,B \in \C^{N\times N}$  left and right eigenvectors of $\BO$  
such that 
\[\begin{aligned} 
\normsp{\BO^{-1}} & \lesssim (\kappa+\eta)^{-1/2}, \qquad &\normsp{\BO^{-1}\cQ}+\norm{B}+\norm{P}& \lesssim 1, &\\ 
\abs{\beta}&\sim\sqrt{\kappa+\eta},\quad &\abs{\braket{P,M\SS[B]B}}& \sim  1, &\qquad \abs{\scalar{P}{B}} \sim 1 ,
\end{aligned}\]
where $\cQ\defeq 1-\cP$ and $\cP\defeq \braket{P,\cdot}B/\braket{P,B}$ are spectral projections of $\BO$.  
\end{enumerate}
\end{proposition}

We now design a suitable norm following \cite{2017arXiv170510661E}. For cumulants of matrix elements $\kappa(w_{ab},w_{cd})$ we use the short-hand notation $\kappa(ab,cd)$. We also use the short-hand notation $\kappa(\vx b,cd)$ for the $\vx=(x_a)_{a\in[N]}$-weighted linear combination $\sum_a x_a \kappa(ab,cd)$ of such cumulants. We use the notation that replacing an index in a scalar quantity by a dot ($\cdot$) refers to the corresponding vector, e.g.~$A_{a\cdot}$ is a short-hand notation for the vector $(A_{ab})_{b\in[N]}$. We fix two vectors $\vx,\vy$ and some large integer $K$ and define the sets
\begin{align*}
I_0&\defeq\set{\vx,\vy}\cup\Set{e_a, P^\ast_{a\cdot} |a\in [N]},\\
 I_{k+1}&\defeq I_k\cup \Set{M\vu|\vu\in I_k}\cup \Set{\kappa_c((M\vu)a,b\cdot),\kappa_d((M\vu)a,\cdot b)|\vu\in I_k,\,a,b\in [N]},
 \end{align*}
where $\kappa_c+\kappa_d=\kappa$ is a decomposition of $\kappa$ according to the Hermitian symmetry\footnote{If $h_{ab}$ is strongly correlated with $h_{cd}$ then, by Hermitian symmetry, it is also strongly correlated with $h_{dc}=\overline{h_{cd}}$. Therefore it is natural to split the covariance into a \emph{direct} and \emph{cross} contribution. The precise splitting $\kappa=\kappa_c+\kappa_d$ is chosen via an optimisation problem; the precise definition is irrelevant for the current proof, see \cite[Remark 2.8]{2017arXiv170510661E} for more details.}. Due to \eqref{cor decay} such a decomposition exists in a way that the operator norms of the matrices $\norm{\kappa_d(\vx a,\cdot b)}$ and $\norm{\kappa_c(\vx a,b\cdot)}$, indexed by $(a,b)$, are bounded uniformly in $\vx$ with $\norm{\vx}\le 1$. We now define the norm
\[\norm{R}_\ast = \norm{R}_\ast^{K,\vx,\vy} \defeq \sum_{0\le k< K}N^{-k/2K} \norm{R}_{I_k} + N^{-1/2} \max_{\vu\in I_K}\frac{\norm{R_{\cdot\vu}}}{\norm{\vu}},\qquad \norm{R}_I \defeq \max_{\vu,\vv\in I} \frac{\abs{R_{\vu\vv}}}{\norm{\vu}\norm{\vv}}. \]
We note that the sets $I_k$ and thereby also the norm $\norm{\cdot}_\ast$ depend implicitly on the spectral parameter $z$ via $M$ and $P$. 
\begin{remark}\label{ast norm apology}
Compared to \cite{2017arXiv170510661E}, the sets $I_k$ contain some additional vectors generated by the vectors of the form $P_{a\cdot}^\ast$ in $I_0$. This addition is necessary to control the spectral projection $\cP$ in the $\norm{\cdot}_\ast$-norm. We note, however, that the precise form of the sets $I_k$ were not important for the proofs in \cite{2017arXiv170510661E}. It was only used that these sets contain deterministic vectors, and that their cardinality grows at most as some finite power $\abs{I_k}\lesssim N^{C_k}$ of $N$.
\end{remark}
In terms of this norm we obtain the following easy estimate on $G-M$ in terms of its projection $\Theta$ 
 onto the unstable direction of  the stability operator $\BO$. 
\begin{proposition}\label{prop stability}
For sufficiently small $\delta$ and fixed $z$ such that $\norm{G-M}_\ast\lesssim N^{-3/K}$ there are deterministic matrices $R_1,R_2$ with norm $\lesssim 1$ such that
\begin{subequations}
\begin{equation}\label{G-M eq}
G-M = \Theta B - \BO^{-1}\cQ[MD] + \mathcal E,\qquad \norm{\mathcal E}_\ast\lesssim N^{2/K} (\abs{\Theta}^2+\norm{D}_\ast^2),\end{equation}
with an error term $\mathcal E$, where $\Theta$, defined in \eqref{def:Theta}, satisfies the approximate quadratic equation
\begin{equation}\label{quadratic eq}
\xi_1\Theta + \xi_2\Theta^2 = \landauO{N^{2/K}\norm{D}_\ast^2+\abs{\braket{R_1D}}+\abs{\braket{R_2D}}}\quad\text{with}\quad\abs{\xi_1}\sim \sqrt{\eta+\kappa},\quad \abs{\xi_2} \sim 1
\end{equation}
and any implied constants are uniform in $\vx,\vy$ and $z\in\DD^\delta$. 
\end{subequations}
\end{proposition}
\begin{proof}
We begin with an auxiliary lemma about the $\norm{\cdot}_\ast$-norm of some important quantities, the proof of which we defer to the appendix.
\begin{lemma}\label{lemma ast norm bounds}
Depending only on the model parameters we have the estimates for any $R\in\C^{N\times N}$,
\begin{equation*}
\norm{M\SS[R]R}_{\ast}\lesssim N^{1/2K} \norm{R}_\ast^2,\qquad \norm{MR}_{\ast}\lesssim N^{1/2K} \norm{R}_\ast,\qquad \norm{\cQ}_{\ast\to\ast}\lesssim 1,\qquad \norm{\BO^{-1}\cQ}_{\ast\to\ast} \lesssim 1.
\end{equation*}
\end{lemma}
Decomposing $G-M=\cP[G-M]+\cQ[G-M]$ and inverting $\BO$ in \eqref{G-M D eq} on the range of $\cQ$ yields
\begin{equation*}
\begin{split}
G-M&=\Theta B+\cQ[G-M]=\Theta B-\BO^{-1}\cQ[MD]+\landauO{N^{1/2K}\norm{G-M}_\ast^2}\\
&=\Theta B-\BO^{-1}\cQ[MD]+\landauO{N^{3/2K}(\abs{\Theta}^2+\norm{D}_\ast^2)},
\end{split}
\end{equation*}
where $\landauO{\cdot}$ is meant with respect to the $\norm{\cdot}_\ast$-norm and the second equality followed by iteration, Lemma \ref{lemma ast norm bounds} and the assumption on $\norm{G-M}_\ast$. Going back to the original equation \eqref{G-M D eq} we find
\begin{equation*}
\beta \Theta B + \BO\cQ[G-M] = -MD+M\SS[\Theta B-\BO^{-1}\cQ[MD]](\Theta B-\BO^{-1}\cQ[MD]) + \landauO{N^{2/K}(\abs{\Theta}^3+\norm{D}_\ast^3)}
\end{equation*}
and thus by projecting with $\cP$ we arrive at the quadratic equation 
\begin{equation*}\begin{split}
  \mu_0 - \mu_1 \Theta +\mu_2\Theta^2&=\landauO{N^{2/K}(\abs{\Theta}^3+\norm{D}_\ast^3)},\qquad \mu_0 = \braket{P,M\SS[\BO^{-1}\cQ[MD]]\BO^{-1}\cQ[MD]-MD}, \\ 
  \mu_1 &= \braket{P,M\SS[B]\BO^{-1}\cQ[MD]+M\SS[\BO^{-1}\cQ[MD]]B}+\beta\braket{P,B},  \qquad \mu_2 =\braket{P,M\SS[B]B}.
\end{split}\end{equation*}
We now proceed by analysing the coefficients in this quadratic equation. We estimate the quadratic term in $\mu_0$ directly by $N^{2/K}\norm{D}_\ast^2$, while we write the linear term as $\braket{R_1D}$ for the  deterministic $R_1\defeq - M^\ast P$ with $\norm{R_1}\lesssim 1$. For the linear coefficient $\mu_1$ we similarly find a deterministic matrix $R_2$ such that $\norm{R_2}\lesssim 1$ and $\mu_1=\braket{R_2D}+\beta\braket{P,B}$. Finally, we find from Proposition \ref{prop stability reference}\eqref{1-CMS bound with flatness} that $\abs{\mu_2}\sim 1$ and $\abs{\beta\braket{P,B}}\sim \sqrt{\kappa+\eta}$. By incorporating the $\abs{\Theta}N^{2/K}$ term into $\xi_2$ we obtain \eqref{quadratic eq}. Here $\delta$ has to be chosen sufficiently small such that Proposition \ref{prop stability reference} is applicable. 
\end{proof}

\subsection{Probabilistic bound}
We now collect the averaged and isotropic bound on $D$ from \cite{2017arXiv170510661E}. We first introduce a commonly used (see, e.g.~\cite{MR3068390}) notion of high-probability bound.
\begin{definition}[Stochastic Domination]\label{def:stochDom}
If \[X=\left( X^{(N)}(u) \,\lvert\, N\in\N, u\in U^{(N)} \right)\quad\text{and}\quad Y=\left( Y^{(N)}(u) \,\lvert\, N\in\N, u\in U^{(N)} \right)\] are families of non-negative random variables indexed by $N$, and possibly some parameter $u$, then we say that $X$ is stochastically dominated by $Y$, if for all $\epsilon, D>0$ we have \[\sup_{u\in U^{(N)}} \P\left[X^{(N)}(u)>N^\epsilon  Y^{(N)}(u)\right]\leq N^{-D}\] for large enough $N\geq N_0(\epsilon,D)$. In this case we use the notation $X\prec Y$. 
\end{definition}
It can be checked (see \cite[Lemma 4.4]{MR3068390}) that $\prec$ satisfies the usual arithmetic properties, e.g.~if $X_1\prec Y_1$ and $X_2\prec Y_2$, then also  $X_1+X_2\prec Y_1 +Y_2$ and  $X_1X_2\prec Y_1 Y_2$.
To formulate the result compactly we also introduce the notations 
\begin{equation}\label{group prec}\begin{split}
 \abs{R}\prec \Lambda\text{ in $\DD$} \,\,&\iff\,\, \norm{R}_\ast^{K,\vx,\vy}\prec \Lambda \text{ unif.~in } \vx,\vy\text{ and $z\in\DD$}, \\ 
  \abs{R}_\av\prec\Lambda\text{ in $\DD$}\,\,&\iff\,\, \frac{\abs{\braket{BR}}}{\norm{B}}\prec\Lambda \text{ unif.~in $B$ and $z\in\DD$}
 \end{split}
\end{equation} 
for random matrices $R=R(z)$ and a deterministic control parameter $\Lambda=\Lambda(z)$, where $B,\vx,\vy$ are deterministic matrices and vectors. We also define an isotropic high-moment norm, already used in \cite{2017arXiv170510661E}, for $p\ge 1$ and a random matrix $R$,
\begin{equation*}
\norm{R}_p \defeq \sup_{\vx,\vy} \frac{\big(\E \abs{\braket{\vx,R\vy}}^p\big)^{1/p}}{\norm{\vx}\norm{\vy}}.
\end{equation*}
\begin{proposition}[Bound on the Error]\label{prop D bound}
Under the Assumptions \ref{assumption A}--\ref{assumption flatness} there exists a constant $C$ such that for any fixed vectors $\vx,\vy$ and matrices $B$ and spectral parameters $z\in\DD^\delta$, and any $p\ge 1$, $\epsilon>0$,
\begin{subequations}
\begin{gather}
\begin{aligned}\label{bootstrapping step}
\frac{\norm{\braket{\vx,D\vy}}_{p}}{\norm{\vx}\norm{\vy}} &\le_{\epsilon,p} N^\epsilon\sqrt{\frac{\norm{\Im G}_{q}}{N\Im z}} \Big(1+\norm{G}_{q}\Big)^{C} \bigg(1+ \frac{\norm{G}_{q}}{N^{\mu}}\bigg)^{Cp} 
\end{aligned}\\
\begin{aligned}
\frac{\norm{\braket{BD}}_{p}}{\norm{B}} &\le_{\epsilon,p} N^{\epsilon}\frac{\norm{\Im G}_{q}}{N\Im z} \Big(1+\norm{G}_{q}\Big)^{C} \bigg(1+ \frac{\norm{G}_{q}}{N^{\mu}}\bigg)^{Cp},
\label{av bound D eq}
\end{aligned}
\end{gather}
where $q\defeq Cp^4/\epsilon$. Here $\mu>0$ depends on $s$ in Assumption \hyperlink{assumpCD}{(CD)}. In particular, if $\abs{G-M}\prec \Lambda\lesssim 1$, then
\begin{equation}\label{D bound eq}
\abs{D}\prec\sqrt{\frac{\varrho+\Lambda}{N\eta}},\qquad
\abs{D}_\av \prec \frac{\varrho+\Lambda}{N\eta}.
\end{equation}
\end{subequations}
\end{proposition}
\begin{proof}
This follows from combining \cite[Theorem 3.1]{2017arXiv170510661E}, the following lemma\footnote{Cf.~Remark \ref{ast norm apology}, where we argue that the proof of \cite{2017arXiv170510661E} about $\norm{\cdot}_\ast$ hold true verbatim in the present case despite the slightly larger sets $I_k$.} from \cite[Lemma 4.4]{2017arXiv170510661E} and $\norm{M}\le M_\ast$.
\end{proof}
\begin{lemma}\label{prec p conversion}
Let $R$ be a random matrix and $\Phi$ a deterministic control parameter. Then the following implications hold:
\begin{enumerate}[(i)]
\item If $\Phi \ge N^{-C}$, $\norm{R} \le N^C$ and $\abs[0]{R_{\vx\vy}}\prec \Phi\norm{\vx}\norm{\vy}$ for all $\vx,\vy$ and some $C$, then $\norm{R}_p\le_{p,\epsilon} N^\epsilon \Phi$ for all $\epsilon>0,p\ge1$. \label{prec to p norm}
\item Conversely, if $\norm{R}_p \le_{p,\epsilon} N^\epsilon \Phi $ for all $\epsilon>0,p\ge 1$, then $\norm{R}_\ast^{K,\vx,\vy} \prec \Phi$ for any fixed $K\in\N$, $\vx,\vy\in\C^N$. \label{p to star norm}
\end{enumerate}
\end{lemma}

\subsection{Bootstrapping}
We now fix $\gamma>0$ and start with the proof of Theorem \ref{theorem local law}. Phrased in terms of the $\norm{\cdot}_\ast$-norm we will prove
\begin{equation}\label{local law}
\abs{G-M} \prec N^{2/K}\left(\sqrt{\frac{\varrho}{N\eta}} + \frac{1}{N\eta}\right), \quad \abs{G-M}_\av \prec N^{2/K} \begin{cases}\frac{1}{N\eta} &\text{$\Re z\in\supp\varrho$}\\ \frac{1}{N(\kappa+\eta)} + \frac{N^{2/K}}{(N\eta)^2\sqrt{\kappa+\eta}} &\text{$\Re z\not\in\supp\varrho$}\end{cases}\quad\text{in}\quad\DD,
\end{equation}
for $\DD=\DD_\gamma^\delta$ and $K\gg 1/\gamma$, i.e.~for $K\gamma$ sufficiently large.  In order to prove \eqref{local law} we use the following iteration procedure. 
\begin{proposition}\label{prop bootstrapping step}
There exists a constant $\gamma_s>0$ depending only on $K$ and $\gamma$ such that \eqref{local law} for $\DD=\DD_{\gamma_0}^\delta$ with $\gamma_0>\gamma$ implies \eqref{local law} also for $\DD=\DD_{\gamma_1}^\delta$ with $\gamma_1\defeq \max\{\gamma,\gamma_0-\gamma_s\}$.
\end{proposition}
\begin{proof}[Proof of \eqref{local law} for $\DD=\DD_\gamma^\delta$, assuming Proposition \ref{prop bootstrapping step}]
For $\DD=\DD_\gamma^\delta$ with $\gamma\ge1$ we have \eqref{local law} by \cite[Theorem 2.1]{2017arXiv170510661E}. For $\gamma<1$ we iteratively apply Proposition \ref{prop bootstrapping step} starting from\footnote{Strictly speaking, in the very first step we start from $\DD^\delta\cap\{\Im z\ge\delta/2\}$ instead of $\DD^\delta_1$ since, depending on the value of $\delta$, the latter might be empty.} $\DD_1^\delta$ finitely many times until we have shown \eqref{local law} for $\DD=\DD_\gamma^\delta$.
\end{proof}
\begin{proof}[Proof of Proposition \ref{prop bootstrapping step}]
We now suppose that \eqref{local law} has been proven for some $\DD=\DD_{\gamma_0}^\delta$ and aim at proving \eqref{local law} for $\DD=\DD_{\gamma_1}^\delta$ for some $\gamma_1=\gamma_0-\gamma_s$, $0<\gamma_s\ll \gamma$. The proof has two stages. Firstly, we will establish the rough bounds
\begin{equation}\label{weak local law}
\abs{\Theta}\prec N^{-5/K}\quad\text{and}\quad\abs{G-M}\prec N^{-5/K}\inD{\gamma_1},
\end{equation}
and then in the second stage improve upon this bound iteratively until we reach \eqref{local law} for $\DD=\DD_{\gamma_1}^\delta$. 

\subsubsection*{Rough bound.}
By \eqref{local law}, Lemma \ref{prec p conversion} and monotonicity of the map $\eta\mapsto\eta\norm{G(\tau+i\eta)}_p$ (see e.g.~(77) in \cite{2017arXiv170510661E}) we find $\norm{G}_p \le_{\epsilon,p} N^{\epsilon+\gamma_s}\le N^{2\gamma_s}$ in $\DD_{\gamma_1}^\delta$. As long as $2\gamma_s<\mu$ we thus have
\begin{equation*}
\norm{D}_p \le_{\epsilon,p} \frac{N^{\epsilon+2C\gamma_s+\gamma_s} }{\sqrt{N\eta}}\le \frac{N^{\gamma_s(2+2C)}}{\sqrt{N\eta}}, \qquad \norm{\braket{BD}}_p \le_{\epsilon,p} \norm{B}\frac{N^{\epsilon+2\gamma_s+2\gamma_s C }}{N\eta} \le \norm{B}\frac{N^{\gamma_s(3+2C)}}{N\eta}.
\end{equation*}
We now fix $\vx,\vy$ and it follows from \eqref{quadratic eq} that 
\begin{equation*}
\abs{\xi_1\Theta + \xi_2\Theta^2 } \prec \frac{N^{2\gamma_s (3+2C)+2/K}}{N\eta} \inD{\gamma_1}
\end{equation*}
and consequently by Lipschitz continuity of the lhs.~with a Lipschitz constant of $\eta^{-2}\le N^2$, and choosing $K,\gamma_s$ large and respectively small enough depending on $\gamma$ we find that with high probability $\abs{\xi_1\Theta+\xi_2\Theta^2}\le N^{-10/K}$ in all of $\DD_{\gamma_1}^\delta$. The following lemma translates the bound on $\abs{\xi_1\Theta+\xi_2\Theta^2}$ into a bound on $\abs{\Theta}$.
\begin{lemma}\label{lemma qu stab}
Let $d=d(\eta)$ be a monotonically decreasing function in $\eta\ge 1/N$ and assume $0\le d\lesssim N^{-\epsilon}$ for some $\epsilon>0$. Suppose that 
\begin{equation*}
\abs{\xi_1\Theta+\xi_2\Theta^2} \lesssim d\quad\text{for all $z \in\DD^\delta$}, \qquad \text{and}\qquad \abs{\Theta} \lesssim \min\left\{ \frac{d}{\sqrt{\kappa+\eta}}, \, \sqrt{d} \right\}\quad\text{for some $z_0$},
\end{equation*}
then also $\abs{\Theta}\lesssim\min\{d/\sqrt{\kappa+\eta},\sqrt d\}$ for all $z'\in\DD^\delta$ with $\Re z'=\Re z_0$ and $\Im z'<\Im z_0$.
\end{lemma}
\begin{proof} This proof is basically identical to the analysis of the solutions to the same approximate quadratic equation, as appeared in various previous works, see e.g.~\cite[Section 9]{MR3699468}. In the spectral bulk this is trivial since then $\abs{\xi_1}\sim \sqrt{\kappa+\eta}\sim1$. Near a spectral edge we observe that $(\kappa+\eta)/d$ is monotonically increasing in $\eta$. First suppose that $(\kappa+\eta)/d\gg 1$ from which it follows that $\abs{\Theta}\lesssim d/\sqrt{\kappa+\eta}\lesssim \sqrt{d}$ in the relevant branch determined by the given estimate on $\Theta$ at $z_0$. Now suppose that below some $\eta$-threshold we have $(\kappa+\eta)/d\lesssim 1$. Then we find $\abs{\Theta}\lesssim \sqrt{\kappa+\eta}+\sqrt{d}\lesssim \sqrt{d}\lesssim d/\sqrt{\kappa+\eta}$ and the claim follows also in this regime.
\end{proof}
Since \eqref{weak local law} holds in $\DD_{\gamma_0}^\delta$ and $1/N\eta\le N^{-100/K}$, we know $\abs{\Theta}\le \min\{ N^{-10/K}/\sqrt{\kappa+\eta}, N^{-5/K} \}$ and therefore can conclude the rough bound $\abs{\Theta}\prec N^{-5/K}$ in all of $\DD_{\gamma_1}^\delta$ by Lemma \ref{lemma qu stab} with $d=N^{-10/K}$. Consequently we have also that 
\begin{equation*}\norm{G-M}_\ast \bm 1(\norm{G-M}_\ast\le N^{-3/K}) \prec N^{-5/K}\inD{\gamma_1}.\end{equation*}
Due to this gap in the possible values for $\norm{G-M}_\ast$ it follows from a standard continuity argument that $\norm{G-M}_\ast \prec N^{-5/K}$ and therefore since $\vx,\vy$ were arbitrary, $\abs{\Theta}\prec N^{-5/K}$ and $\abs{G-M}\prec N^{-5/K}$ in all of $\DD_{\gamma_1}^\delta$. 

\subsubsection*{Strong bound.}
All of the following bounds hold uniformly in the domain $\DD_{\gamma_1}^\delta$ which is why we suppress this qualifier. By combining Propositions \ref{prop stability} and \ref{prop D bound} we find for deterministic $0\le\theta\le\Lambda\le N^{-3/K}$ under the assumptions $\abs{\Theta}\prec\theta$, $\abs{G-M}\prec\Lambda$,  that
\begin{equation}\label{self-improving bound}
\abs{G-M} \prec \theta + N^{2/K} \sqrt{\frac{\varrho+\Lambda}{N\eta}},\qquad \abs{\xi_1\Theta+\xi_2\Theta^2}\prec  N^{2/K}\frac{\varrho+\Lambda}{N\eta}.
\end{equation}
The bound on $\abs{G-M}$ in \eqref{self-improving bound} is a self-improving bound and we find after iteration that
\begin{equation*}
\abs{G-M} \prec \theta+N^{2/K}\left(\frac{1}{N\eta}+\sqrt{\frac{\varrho+\theta}{N\eta}}\right),\quad\text{hence}\quad \abs{\xi_1\Theta+\xi_2\Theta^2}\prec  N^{2/K}\frac{\varrho+\theta}{N\eta}+N^{4/K}\frac{1}{(N\eta)^2}.
\end{equation*}
We now distinguish whether $\Re z$ is inside or outside the spectrum. Inside we have $\varrho\sim \sqrt{\kappa+\eta}$, so we fix $\theta$ and use Lemma \ref{lemma qu stab} with $d = N^{2/K}(\sqrt{\kappa+\eta} +\theta) /(N\eta)+N^{4/K}/(N\eta)^2$ to conclude $\abs{\Theta} \prec \min\{d/\sqrt{\kappa+\eta}, \sqrt{d}\}$ from the input assumption $\abs{\Theta}\prec N^{2/K}/N\eta$ in $\DD_{\gamma_0}$. 
Iterating this bound, we obtain
\begin{equation*}
\abs{\Theta}\prec  N^{2/K}\frac{1}{N\eta},\quad\text{hence}\quad \abs{G-M} \prec N^{2/K}\left(\sqrt{\frac{\varrho}{N\eta}}+\frac{1}{N\eta}\right).
\end{equation*}
By an analogous argument, outside of the spectrum we have an improved bound on $\Theta$
\begin{equation*}
\abs{\Theta} \prec N^{2/K} \frac{1}{N(\kappa+\eta)} + N^{4/K} \frac{1}{(N\eta)^2\sqrt{\kappa+\eta}},
\end{equation*}
because $\varrho\sim\eta/\sqrt{\kappa+\eta}$. Finally, for the claimed bound on $\abs{G-M}_\av$ we use \eqref{G-M eq} in order to obtain a bound on $\abs{G-M}_\av$ in terms of a bound on $\Theta$. 
\end{proof}
Due to \eqref{local law}, we now have all the ingredients to prove the local law, as well as delocalisation of eigenvectors, and the absence of eigenvalues away from the support of $\varrho$.
\begin{proof}[Proof of Theorem \ref{theorem local law}, Corollary \ref{cor no eigenvalues outside} and Corollary \ref{cor delocalisation}]
The local law inside the spectrum \eqref{iso bulk}--\eqref{av bulk} follows immediately from \eqref{local law}. Now we prove Corollary \ref{cor no eigenvalues outside}. If there exists an eigenvalue $\lambda$ with $\dist(\lambda,\supp\varrho)>N^{-2/3+\omega}$, then at, say, $z=\lambda+\ii N^{-4/5}$ we have $\abs{\braket{G-M}}\ge c N^{-1/5}$. On the other hand we know from the improved local law \eqref{local law} that with high probability $\abs{\braket{G-M}}\le N^{-1/4}$ and we obtain the claim.

We now turn to the proof of Corollary \ref{cor delocalisation}. For the eigenvectors $\vu_k$ and eigenvalues $\lambda_k$ of $H$ we find from the spectral decomposition and the local law with high probability
\[  1\gtrsim \Im \langle \vx , G \vx \rangle= \eta\sum_k \frac{\abs{\langle \vx, \vu_k\rangle }^2}{(\tau-\lambda_k)^2+\eta^2}\ge \frac{\abs{\langle \vx, \vu_k\rangle}^2}{\eta}\quad\text{for}\quad z=\tau+\ii\eta\]
 for any normalised $\vx \in \C^N$, 
where the last inequality followed assuming that $\tau$ is chosen $\eta$-close to $\lambda_k$. With the choice $\eta = N^{-1+\gamma}$ for arbitrarily small $\gamma>0$ the claim follows. Note that for this proof only \eqref{iso bulk} of Theorem~\ref{theorem local law} was used.

Finally, we establish \eqref{av outisde} and consider $z\in\DD^\delta$ with $\dist(\Re z,\supp\varrho)\ge N^{-2/3+\omega}$ and $\vx,\vy,B$ fixed. 
As in the proof of \cite[Corollary 1.11]{MR3719056}, the optimal local law \eqref{local law} implies rigidity 
up to the edge as formulated in Corollary~\ref{cor rigidity}. The only difference is
that this standard argument proves \eqref{eq rigidity} only if the supremum is restricted to $\tau\in\supp\varrho$ with $\dist(\tau,\partial\supp\varrho) \ge N^{-2/3+\epsilon}$.
The cause for this restriction is a possible mismatch of the labelling of the edge eigenvalues, in other words the precise
location of  $N^\epsilon$ eigenvalues near an internal gap is not established yet; they may belong to either band 
adjacent to this gap. This shortcoming will be remedied  by the band rigidity in the proof of Corollary \ref{cor rigidity} in Section~\ref{sec:band_rigidity} below. 
However, for the current argument, the imprecise
location of $N^\epsilon$ eigenvalues does not matter. In fact, already from this version of rigidity, together with 
 the delocalisation of eigenvectors (Corollary \ref{cor delocalisation}) and the absence of eigenvalues
outside of the spectrum by Corollary \ref{cor no eigenvalues outside} we have, at $z=\tau+\ii\eta$ (recall that we consider $z\in\DD^\delta$ with $\dist(\Re z,\supp\varrho)\ge N^{-2/3+\omega}$),
\[
    \Im \braket{ \vx , G(z) \vx}= \eta\sum_k \frac{\abs{\langle \vx, \vu_k\rangle }^2}{(\tau-\lambda_k)^2+\eta^2} \prec \frac{1}{N}\sum_k \frac{\eta}{(\tau-\lambda_k)^2+\eta^2}
    \prec  \int_\R \frac{\eta\,\varrho(x) \diff x}{\abs{\tau- x}^2+\eta^2}
\]
 for any normalised vector $\vx$. From the square
root behaviour of $\varrho$ at the edge and $\kappa(z) \ge N^{-2/3+\omega}$ we can easily infer  $\norm{\Im G}_\ast \prec \eta/\sqrt{\kappa+\eta}$.  Therefore it follows from Proposition \ref{prop D bound} that $\norm{D}_\ast^2+\abs{\braket{RD}}\prec 1/(N\sqrt{\kappa+\eta})$ and from \eqref{quadratic eq} and Lemma \ref{lemma qu stab} that $\abs{\Theta}\prec N^{2/K-1}/(\kappa+\eta)$. Finally, we thus obtain, \begin{equation*}
\abs{G-M}_\av \prec \frac{N^{2/K}}{N(\kappa+\eta)}+ \frac{N^{2/K}}{N\sqrt{\kappa+\eta}}\lesssim N^{2/K}\frac{1}{N(\kappa+\eta)}
\end{equation*}
from \eqref{G-M eq} and \eqref{av outisde} follows. 
\end{proof}

\section{Analysis of the Matrix Dyson equation} \label{sec:dens_close_to_edge} 

The essential prerequisite for edge universality is the regularity of the edge, i.e.~the local square root behavior of the self consistent density $\dens$ 
as imposed in Definition \ref{def regular edge}. For the proof of universality via \cite{2017arXiv171203881L}, 
however, it is necessary to first establish that the square-root behaviour and the adjacent gap persist in a macroscopic interval. 
This is achieved in the following main theorem whose proof will be given  
 in Section~\ref{subsec:proof_dens_close_to_regular_edge_below} after several preparatory results. 
In particular, as a second main result of this section, in Theorem~\ref{thm:B_inverse_sharp_edge}, 
we will give a sharp estimate on the inverse of the stability operator $\BO = \Id - M \SS[\genarg] M$ 
which also plays a central role in the proof of the local law in Section~\ref{sec proof local law}. 

\begin{theorem}[Behaviour of $\dens$ close to a square root edge] \label{thm:dens_close_to_regular_edge}
Let \ref{assumption A}, \ref{assumption flatness} and \ref{assumption bounded M} be satisfied for some $\tau_0 \in \R$. 
If $\tau_0\in\partial\supp\varrho$ is a regular edge then there are $c \sim 1$ and $\delta_* \sim 1$ such that 
\[ 
\dens(\tau_0 + \omega) = \begin{cases} 
c \abs{\omega}^{1/2} + \ord(\abs{\omega}), \quad & \text{if } \omega \in [-\delta_*,0], \\  
0, \qquad  & \text{if } \omega \in [0,\delta_*].
\end{cases} 
\] 
\end{theorem} 

In this section and, in particular, the previous theorem, the comparison relation $\sim$ is understood with respect to the 
constants in \ref{assumption A}, \ref{assumption flatness} and \ref{assumption bounded M} 
as well as in \eqref{sqrt growth}.  

We now outline the strategy for the proof of Theorem~\ref{thm:dens_close_to_regular_edge}. 
First, we will extend $M$ to the real line by showing that it is 
$1/2$-Hölder continuous in the vicinity of $\tau_0$ (see Corollary~\ref{cor:Hoelder_continuity_M} below). 
The Hölder continuity also yields an a-priori bound on $\Delta \defeq M(\tau_0 + \omega) - M(\tau_0)$,   hence
on  $\varrho(\tau_0+\omega) = \pi^{-1}\avg{\Im M(\tau_0 + \omega)}= \pi^{-1}\avg{\Im \Delta}$ as well, 
 with  small 
 $\omega \in \R$. 
Second, by using this bound, we will verify that $\Delta$ is governed by a scalar quantity analogous to $\Theta$ from \eqref{def:Theta} which satisfies a quadratic 
equation (see Proposition~\ref{pro:quadratic_for_shape_analysis} below).    The fact that $\Im \Delta\ge 0$ will select the correct solution 
to this quadratic equation and Theorem~\ref{thm:dens_close_to_regular_edge} 
will follow from analysing the  stability of this solution. 

 The  equation for $\Delta$ can be obtained from subtracting the MDE at $\tau_0 + \omega$ and $\tau_0$. It reads as
\begin{equation} \label{eq:stab_equation_non_symm}
 \BO[\Delta] = M \SS[\Delta] \Delta + \omega M^2  + \omega M \Delta, \qquad M = M(\tau_0).
\end{equation}
To express $\Delta$ from \eqref{eq:stab_equation_non_symm}
 it is therefore essential to understand 
the  instabilities of  $\BO^{-1}$ very precisely. 
 The main difficulty is that near the edge  $\BO$ has a
 small eigenvalue that is very sensitive to a delicate balance between $\SS$ and $M$. 
 An additional complication is that $\BO$ is non-selfadjoint. Both obstacles are 
overcome by representing 
$\BO$ in the form $\BO = \cV (\mathcal{U}  - \cF)\cV^{-1}$, where $\mathcal{U}$ is unitary,  $\cV$ is bounded invertible,
$\cF$ is self-adjoint and it preserves the cone of positive matrices.  Thus a Perron-Frobenius argument can be applied to $\cF$, i.e.~
 its norm can be obtained  simply by finding its top eigenvector. 
 In this way we can very precisely determine the size of 
 $M\SS[\cdot ]M$ and estimate its top eigenvalue without explicitly solving the MDE. 
 This representation of $\BO$ (cf.~\eqref{eq:representation_BO} below) with the  Perron-Frobenius argument 
  is one  of the main results of \cite{2016arXiv160408188A} and
 the analysis of $\cF$ will partly be imported from \cite{2016arXiv160408188A}. 
We will see that $\BO^{-1}$ has precisely one unstable direction and
 we will obtain the quadratic equation  for $\Theta$, the projection of $\Delta$, onto this  direction.
The sharp estimate on the eigenvalue of the unstable direction will give rise to the following bound 
on $\BO^{-1}$. 

\begin{theorem}[Sharp bound on $\BO^{-1}$ near a regular edge]  \label{thm:B_inverse_sharp_edge} 
Let \ref{assumption A}, \ref{assumption flatness} and \ref{assumption bounded M} be satisfied for a regular edge $\tau_0\in\pt\supp\dens$. 
Then there is $\delta_* \sim 1$ such that we have 
\[ \normsp{\BO(z)^{-1}} + \norm{\BO(z)^{-1}} \lesssim \frac{1}{ \dens(z) + \eta \dens(z)^{-1}}, \] 
for all $z \in \Hb$ satisfying $\abs{z- \tau_0} \leq \delta_*$, where $\eta = \Im z$. 
\end{theorem} 

From the previous theorem, we will immediately conclude the $1/2$-Hölder continuity stated in the following corollary. 
The proofs of both statements will be given in Section~\ref{subsec:proof_hoelder_continuity} below.  

\begin{corollary}[Hölder-continuity of $M$] \label{cor:Hoelder_continuity_M}
Let \ref{assumption A}, \ref{assumption flatness} and \ref{assumption bounded M} be satisfied for a regular edge $\tau_0\in\R$.
Then $M$ is uniformly $1/2$-Hölder continuous around $\tau_0$ in the sense that there is $\delta_* \sim 1$ such that 
\[ \norm{M(z_1)- M(z_2)} \lesssim \abs{z_1 - z_2}^{1/2} \] 
for all $z_1, z_2 \in \{ \tau + \ii \eta \colon \abs{\tau-\tau_0} \leq \delta_*, ~ 0 < \eta < \infty\}$. In particular, $M$ has a unique extension to $[\tau_0 -\delta_*, \tau_0 + \delta_*]$. 
\end{corollary}

\subsection{Analysis of the stability operator} 

In this section, we will always assume that \ref{assumption A}, \ref{assumption flatness} and \ref{assumption bounded M} are satisfied for some $\tau_0 \in \R$.  
The main result of this section is the bound on the inverse of the stability operator $\BO$ in Proposition~\ref{pro:B_inverse_improved_bound} below. 
We introduce the \emph{balanced polar decomposition} 
\begin{equation} \label{eq:balanced_polar_decomposition}
 M = Q^* U Q, 
\end{equation}
where we define
\begin{equation} \label{eq:def_q_u}
 W \defeq (\Im M)^{-1/2} (\Re M) (\Im M)^{-1/2} + \ii\id , \qquad Q \defeq \abs{W}^{1/2} (\Im M)^{1/2}, \qquad U \defeq\frac{W}{\abs{W}}.
\end{equation}
We remark that $W$ is normal, $\abs{W} \defeq (W^*W)^{1/2}$, $U$ is unitary and $\Im U$ is 
positive definite. 
In this context, the balanced polar decomposition first appeared in \cite{2016arXiv160408188A}. 
We also define 
\begin{equation} \label{eq:def_S_F_U_Sigma}
S \defeq \sign \Re U, \qquad F_U \defeq \dens^{-1} \Im U, \qquad \sigma \defeq \avg{SF_U^3}. 
\end{equation}

The quantities $\BO$, $W$, $Q$, $U$, $S$, $F_U$ and $\sigma$ introduced above all depend on $z$ through the $z$-dependence of $M$. In the following, we will 
mostly omit this dependence from our notation. 

\begin{proposition}[General bound on $\BO^{-1}$]  \label{pro:B_inverse_improved_bound}
If \ref{assumption A}, \ref{assumption flatness} and \ref{assumption bounded M} are satisfied for some $\tau_0\in \R$ then, uniformly for all $z \in \Dbdd$, we have
\begin{equation} \label{eq:B_inverse_improved_bounds}
 \normsp{\BO(z)^{-1}} + \norm{\BO(z)^{-1}}\, \lesssim \, \frac{1}{\dens(z)(\dens(z) + \abs{\sigma(z)}) + \eta\dens(z)^{-1}}, \qquad 
\quad \eta = \Im z. 
\end{equation}
\end{proposition}

This proposition will be shown at the end of the present section.  
Now, we apply it to show that $M$ is $1/3$-Hölder continuous. 

\begin{corollary}[$1/3$-Hölder continuity of $M$] \label{cor:hoelder_1_3}
Let \ref{assumption A}, \ref{assumption flatness} and \ref{assumption bounded M} be satisfied for some $\tau_0 \in \R$.  
Then the solution $M$ of the MDE, \eqref{MDE}, is uniformly $1/3$-Hölder continuous around $\tau_0$ in the sense that, 
for each $\theta\in (0, \omega_*)$, we have 
\[ \norm{M(z_1)- M(z_2)} \lesssim_\theta \abs{z_1 - z_2}^{1/3} \] 
for all $z_1, z_2 \in \{ \tau + \ii \eta \colon \abs{\tau- \tau_0} \leq \omega_* - \theta, ~ 0 < \eta < \infty \}$.    
\end{corollary} 

Before we prove the previous corollary, we state and prove the following lemma. 
It collects a few basic properties of $M$, $Q$ and $U$ which will often be used in the following.  

\begin{lemma}[Properties of $M$, $Q$ and $U$] \label{lem:prop_m_q}
Uniformly for $z \in \Dbdd$, we have 
\begin{subequations}
\begin{align}
\norm{M(z)^{-1}} & \sim \norm{M(z)} \sim 1, \label{eq:M_inverse_bounded}\\ 
\Im M(z) & \sim \avg{\Im M(z)}, \label{eq:im_M_sim_avg} \\ 
\norm{Q(z)} & \sim \norm{Q(z)^{-1}} \sim 1, \label{eq:q_sim_1}\\ 
\Im U(z) & \sim \avg{\Im U(z)} \sim \dens(z), \label{eq:im_u_sim_rho} 
\end{align} 
\end{subequations}
where $A\lesssim B$ and $A\sim B$ for matrices $A,B$ indicate that $A\le C B$ and $cB\le A\le CB$ for some constants $c,C$ in the sense of quadratic forms. 
\end{lemma} 

\begin{proof}[Proof of Lemma~\ref{lem:prop_m_q}]
The bounds in \eqref{eq:M_inverse_bounded} and \eqref{eq:im_M_sim_avg} follow easily  from the bound on $\norm{M}$ on $\Dbdd$ as well as the flatness of $\SS$ 
(see e.g.~the proof of Proposition 4.2 in \cite{2016arXiv160408188A}).  

For the proof of \eqref{eq:q_sim_1}, we use the monotonicity of the square root and \eqref{eq:im_M_sim_avg} to obtain 
\[\begin{aligned} 
 Q^* Q & = (\Im M)^{1/2} ( 1 + (\Im M)^{-1/2} (\Re M) (\Im M)^{-1} (\Re M) (\Im M)^{-1/2})^{1/2} (\Im M)^{1/2} \\ 
 & \sim \avg{\Im M}^{-1/2}(\Im M)^{1/2} \Big( (\Im M)^{-1/2} ( (\Im M)^2 + (\Re M)^2) (\Im M)^{-1/2} \Big)^{1/2} (\Im M)^{1/2}. 
\end{aligned} \] 
Thus, employing $(\Re M)^2 + (\Im M)^2 \sim 1$ by \eqref{eq:M_inverse_bounded} yields \eqref{eq:q_sim_1} due to \eqref{eq:im_M_sim_avg}. 

Owing to \eqref{eq:q_sim_1}, \eqref{eq:im_u_sim_rho} is a direct consequence of \eqref{eq:im_M_sim_avg}. 
This completes the proof of Lemma~\ref{lem:prop_m_q}. 
\end{proof}

In the following, we will use the derivative of $M$ with respect to $z$ several times. 
For $z \in \Hb$, we take the derivative of \eqref{MDE} with respect to $z$. Owing to the invertibility of $\BO=\BO(z)$, this yields 
\begin{equation} \label{basic MDE equations}
 \pt_z M(z) = \BO^{-1}[M(z)^2] 
\end{equation}
for $z \in \Dbdd$.

\begin{proof}[Proof of Corollary~\ref{cor:hoelder_1_3}]
As $\pt_z \Im M(z) = (2\ii)^{-1} \pt_z M(z)$ due to the analyticity of $M$, we conclude from \eqref{basic MDE equations} and \eqref{eq:B_inverse_improved_bounds} and \eqref{eq:im_M_sim_avg} that 
\[ \norm{\pt_z \Im M(z)} \lesssim \dens(z)^{-2} \sim \norm{\Im M(z)}^{-2}. \] 
This implies that $z \mapsto (\Im M(z))^3$ is Lipschitz-continuous on $\Dbdd$. 
Therefore, $\Im M(z)$ is $1/3$-Hölder continuous on $\Dbdd$ (see e.g.~Theorem~X.1.1 in \cite{MR1477662}) and, thus, 
$M$ is uniformly $1/3$-Hölder continuous on $\{ \tau + \ii \eta \colon \abs{\tau - \tau_0} \leq \omega_* - \theta,~0 < \eta < \infty \}$ 
for all $\theta\in (0,\omega_*)$ (see e.g.~Lemma~A.7 in \cite{2015arXiv150605095A} as well as Lemma~A.1 in \cite{shapepaper} for a slightly 
more general formulation). 
\end{proof} 

For the analysis of the stability operator $\BO$ defined in \eqref{eq:def_BO}, 
we now introduce the Hermitian operator $\cF\colon \C^{N\times N} \to \C^{N\times N}$ defined through
\begin{equation} \label{eq:def_F}
 \cF \defeq \cC_{Q,Q^*}\SS \cC_{Q^*,Q}. 
\end{equation}

Here, we used the following notation for operators on $\C^{N\times N}$. For $T_1, T_2 \in \C^{N\times N}$, we define 
the operator $\cC_{T_1, T_2}\colon \C^{N\times N} \to \C^{N\times N}$ through 
\[ \cC_{T_1,T_2}[R] = T_1 R T_2 \] 
for all $R \in \C^{N\times N}$. We also set $\cC_T \defeq \cC_{T,T}$.  
The importance of $\cF$ for the analysis of $\BO$ and its inverse comes from the following consequence of the balanced polar decomposition \eqref{eq:balanced_polar_decomposition}:
\begin{equation} \label{eq:representation_BO}
 \BO=\Id - \cC_M \SS = \cC_{Q^*,Q} \cC_U ( \cC_U^* - \cF) \cC_{Q^*,Q}^{-1}. 
\end{equation}

When $\dens=\dens(z)$ is small, we will view  $\BO$  as a perturbation of the operator $\BO_0$, which we introduce now. We define 
\begin{equation} \label{eq:def_B_0_E}
\BO_0 \defeq \cC_{Q^*,Q}( \Id - \cC_S \cF) \cC_{Q^*,Q}^{-1}, \qquad \cE \defeq 
(\cC_{Q^*SQ} - \cC_M)\SS=\cC_{Q^*,Q}(\cC_S - \cC_U)\cF\cC_{Q^*,Q}^{-1},  
\end{equation}
with $U$ and $Q$ defined in \eqref{eq:def_q_u}, $S$ defined in \eqref{eq:def_S_F_U_Sigma} and $\cF$ defined in \eqref{eq:def_F}.
Note $\BO_0 = \Id - \cC_{Q^*SQ}\SS$, i.e.~in the definition of $\BO$, the unitary matrix $U$ in $M=Q^*UQ$ is replaced by $S$. 
Thus, we have $\BO = \BO_0 + \cE$. 

In the following, we will often use \eqref{eq:q_sim_1} and \eqref{eq:im_u_sim_rho}. 
In particular, since $\id  - \abs{\Re U } = \id  - \sqrt{\id  - (\Im U)^2} \leq (\Im U)^2 \lesssim \dens^2$, we also obtain  
\begin{equation} \label{eq:Re_u_s}
 \Re U = S + \ord(\dens^2), \qquad \Im U \,=\, \ord(\dens)\,,\qquad \Re M = Q^* S Q + \ord(\dens^2) 
\end{equation}
and with $\cC_S-\cC_U =\ord(\norm{S-U})= \ord(\dens)$ we get 
\begin{equation} \label{eq:E=ord rho}
\cE\, =\,\ord(\dens)\,.
\end{equation}
Here, we use the notation $\mathcal{R} = \mathcal{T} + \ord(\alpha)$ for operators $\mathcal{R}$ and $\mathcal{T}$ 
on $\C^{N\times N}$ and $\alpha>0$ if $\norm{\mathcal{R} - \mathcal{T}} \lesssim \alpha$.
By the functional calculus, the normal matrices $U$, $\Re U$, $S$ and $F_U$ commute. Hence, $\cC_S[F_U] = F_U$. 

The MDE, \eqref{MDE}, the balanced polar decomposition, $M = Q^* U Q$, and the definition of $\cF$ in \eqref{eq:def_F} yield
\begin{equation} \label{eq:dyson_second_version}
 - U^* = Q(z-A)Q^* + \cF[U].
\end{equation}
We take the imaginary part of \eqref{eq:dyson_second_version} and use \eqref{eq:q_sim_1} 
 as well as \eqref{eq:im_u_sim_rho} to conclude that 
\begin{equation} \label{eq:F_f_u}
 (\Id - \cF) [F_U] = \eta \dens^{-1} QQ^* = \ord(\eta\dens^{-1}). 
\end{equation}

We also introduce the operator $\BO_*$, and view it as a perturbation of $\BO_0$, via 
\[ \BO_* \defeq \Id - \cC_{M^*,M} \SS, \qquad \cE_* \defeq  
(\cC_{Q^*SQ} - \cC_{M^*,M})\SS=\cC_{Q^*,Q}(\cC_S - \cC_{U^*,U})\cF\cC_{Q^*,Q}^{-1}.  \] 
Hence, we have $\BO_* = \BO_0 + \cE_*$. 
Analogously to \eqref{eq:E=ord rho}, we conclude from \eqref{eq:Re_u_s} that 
\begin{equation} \label{eq:E_star_ord_rho}
 \cE_* = \ord(\dens). 
\end{equation}

In the following, for $z\in \C$ and $\eps>0$, we denote by $D_\eps(z) \defeq \{ w \in \C \colon \abs{z-w} < \eps\}$ the disk in $\C$ of radius $\eps$ around $z$.

\begin{lemma}[Spectral properties of stability operator for small density] \label{lem:prop_F_small_dens} 
Let $\TO \in \{\Id-\cF, \Id-\cC_S \cF,\BO_0,\BO,\BO_*\}$.  Then there are $\dens_* \sim 1$ and $\eps \sim 1$ such that
\begin{equation} \label{eq:B_0_resolvent_bound}
 \normsp{( \TO-\omega \Id)^{-1}} +\norm{( \TO-\omega \Id)^{-1}} +\norm{(  \TO^*-\omega \Id)^{-1}} \lesssim 1 
\end{equation}
uniformly for all $z \in\Dbdd$ satisfying $\dens(z)+\eta \dens(z)^{-1}\leq \dens_*$ and for all $\omega \in \C$ with $\omega \not \in D_{\eps}(0) \cup D_{1-2 \eps}(1) $. 
Furthermore, there is a single simple (algebraic multiplicity $1$) eigenvalue $\lambda$ in the disk around $0$, i.e.~
\begin{equation}
\label{eq:nondegeneracy for CF}
\spec(\TO) \cap D_{\eps}(0)\,=\, \{\lambda\}\quad \text{and} \quad \rank \cP_\TO\,=\, 1\,, \quad \text{where} \quad 
\cP_\TO\,\defeq \, -\frac{1}{2\pi\ii}\int_{\partial D_\eps(0)}(\TO-\omega\Id)^{-1} \di\omega \,.
\end{equation}
\end{lemma}

\begin{proof} First, we introduce the bounded operators $\cV_t\colon \C^{N\times N} \to \C^{N\times N}$ for $t \in [0,1]$ interpolating between $\Id$ and $\cC_S$ by
\[ \cV_t \defeq  (1-t) \Id + t \cC_S\,.\]
We will perform the proof one by one for the choices $\TO=\Id-\cF,\Id-\cV_t\cF ,\BO_0,\BO,\BO_*$ in that order. 
We will first show that the operator $\Id - \cF$ has a spectral gap above the single eigenvalue around 0, so for this choice the statements are easy. Then we perform two approximations. First, we interpolate between $\Id-\cF$ and $\Id - \cC_S\cF$
via $\Id - \cV_t \cF$. This gives Lemma \ref{lem:prop_F_small_dens} for $\TO = \BO_0$. 
Then we use perturbation theory to get the results for $\TO = \BO = \BO_0 + \ord(\dens)$ and for $\TO = \BO_* =\BO_0 + \ord(\dens)$.  
Note that for all these choices of $\TO$ the bound $\normtwoinf{\Id-\TO} \lesssim 1$ holds due to $\normtwoinf{\SS} \lesssim 1$, $\norm{M} \lesssim 1$ 
and \eqref{eq:q_sim_1}. Hence, the invertibility of $\TO-\omega \Id$ as an operator on $(\C^{N\times N}, \norm{\genarg})$ and on $(\C^{N\times N}, \normtwo{\genarg})$   are therefore closely related as 
\[ \norm{(\TO-\omega \Id)^{-1}} \leq \abs{1-\omega}^{-1} ( 1 + \normtwoinf{\Id - \TO} \normsp{(\TO-\omega \Id)^{-1}} ). \] 
The proof of this bound is elementary, see e.g.~Lemma~B.2 (ii) of \cite{shapepaper}.  
In particular, it suffices to show \eqref{eq:nondegeneracy for CF} and the $\normsp{\,\cdot\,}$-norm bound
\begin{equation} \label{eq:T 2 norm bound}
 \normsp{( \TO-\omega \Id)^{-1}} \,\lesssim\, 1\,,
\end{equation}
for $\omega \not \in D_{\eps}(0) \cup D_{1-2 \eps}(1) $  in \eqref{eq:B_0_resolvent_bound} to establish the lemma. 
For $\TO=\Id-\cF$ both of these assertions are true due to the following facts about the operator $\cF$ 
that have been the backbone of the analysis of \cite{2016arXiv160408188A}:   

\begin{enumerate}[(a)]
\item \label{item:F_simple_eigenvalue} The norm $\normsp{\cF}$ of the Hermitian operator $\cF \colon \C^{N\times N} \to \C^{N\times N}$ is a simple eigenvalue of $\cF$. 
Moreover, there is a unique, positive definite eigenvector 
$F \in \C^{N\times N}$ such that $\cF[F] = \normsp{\cF} F$ and $\normtwo{F} = 1$. This eigenvector satisfies 
\begin{equation} \label{eq:normtwo_F}
 1- \normsp{\cF} = (\Im z) \frac{\scalar{F}{QQ^*}}{\scalar{F}{\Im U}}. 
\end{equation}
In particular, $\normsp{\cF}\leq 1$. 
\end{enumerate}
Furthermore, uniformly for all $z \in \Dbdd$, the following properties hold true: 
\begin{enumerate}[(a)]
\addtocounter{enumi}{1} 
\item The eigenvector $F$ is bounded from above and below, i.e.~
\begin{equation}
F \sim 1.  \label{eq:f_sim_norm_m} 
\end{equation}
\item \label{item:F_spectral_gap} The operator $\cF$ has a spectral gap $\vartheta\sim 1$, i.e.~ 
\begin{equation}
 \spec(\cF/\normsp{\cF}) \subset [-1 + \vartheta, 1- \vartheta] \cup \{ 1\}.\label{eq:spec_F}
\end{equation}
\item \label{item:eigenvector_approx} The eigenvector $F$, $\cF F= \normsp{\cF}F$, satisfies 
\begin{equation} \label{eq:f_u_approx_f}
F  = \normtwo{F_U}^{-1} F_U+\ord(\eta \dens^{-1})\,,
\end{equation}
\end{enumerate}
These facts are proven as Lemma~4.7 in \cite{2016arXiv160408188A} 
using Lemma~\ref{lem:prop_m_q} instead of (4.11) and (4.23) in the proof of (4.33) in~\cite{2016arXiv160408188A}. 
Moreover, the proof of \eqref{eq:f_u_approx_f} follows from \eqref{eq:F_f_u} and $\normsp{\cF}= 1 +\ord(\eta\dens^{-1})$ (cf.~\eqref{eq:normtwo_F}) by straightforward perturbation theory of the simple isolated eigenvalue~$\normsp{\cF}$.

Now we consider the choice $\TO=\TO_t=\Id-\cV_t \cF$. Once \eqref{eq:T 2 norm bound}, and with it \eqref{eq:B_0_resolvent_bound}, is established for $\TO_t$, the statement about the single isolated eigenvalue \eqref{eq:nondegeneracy for CF} follows. 
Indeed, assuming \eqref{eq:B_0_resolvent_bound} for $\TO=\TO_t$, we obtain that $\TO_t$ and, hence, the rank of $\cP_{\TO_t}$ is a 
continuous function of $t$ on $[0,1]$. Hence, the rank of $\cP_{\TO_t}$ is constant along this interpolation. On the other hand, 
$\rank\cP_{\TO_0} = 1$ by Fact \eqref{item:F_simple_eigenvalue} above. Therefore, for each $t \in [0,1]$, $\spec(\TO_t) \cap D_\eps(0)$ consists of precisely one simple eigenvalue. 
We are thus left with establishing \eqref{eq:T 2 norm bound} for $\TO_t$.
As $\normsp{\cV_t} \leq 1$ and $\normsp{\cF} \le 1$ the bound \eqref{eq:T 2 norm bound} is certainly satisfied for $\abs{\omega} \ge 3$. Thus, we now assume $\abs{\omega}\le 3$. 
In order to conclude \eqref{eq:T 2 norm bound}, we now show a lower bound on $\normtwo{( (1-\omega)\Id - \cV_t \cF)[R]}$ for all normalized, $\normtwo{R}=1$, elements $R \in \C^{N\times N}$. 
 We decompose $R$ as $R = \alpha F + R^\perp$, where $R^\perp \perp F$ with respect to the Hilbert-Schmidt scalar product on $\C^{N \times N}$  and $\alpha \in \C$. Then
 \begin{equation} \label{eq:estimate_T_omega} 
\begin{aligned} 
\normtwo{( (1-\omega)\Id-\cV_t \cF  )[R]}^2 = \, \abs{\alpha}^2 \abs{\omega}^2+\normtwo{( (1-\omega)\Id-\cV_t \cF )[R^\perp]}^2 +\ord\big(\eta \dens^{-1}\big)  \,,
\end{aligned}
\end{equation} 
because of $\normsp{\cF}= 1 +\ord(\eta\dens^{-1})$, $\cV_t[F_U] = F_U$ together with \eqref{eq:f_u_approx_f}, and because the mixed terms are negligible due to 
\[
\scalar{F}{\cV_t \cF [R^\perp]}\,=\, \scalar{\cF\cV_t[F]}{R^\perp}\,=\, \ord(\normtwo{R^\perp}\eta\dens^{-1})\,.
\]
Using the spectral gap   $ \vartheta \sim 1 $ of $\cF$ from \eqref{eq:spec_F} and $R^\perp \perp F$ we infer \eqref{eq:T 2 norm bound} from \eqref{eq:estimate_T_omega} by estimating 
\[
\normtwo{( (1-\omega)\Id-\cV_t \cF )[R^\perp]}^2 \,\ge\, \dist(\omega, D_{1-\vartheta}(1))^2\normtwo{R^\perp}^2 \,\ge\,  (\vartheta-2\eps)^2(1-\abs{\alpha}^2), 
\] 
optimizing in $\alpha$ and choosing $\eps \leq \vartheta/3$. 
This shows the lemma for $\TO=\Id-\cV_t \cF$.

Since $\BO_0$ is related by the similarity transform \eqref{eq:def_B_0_E} to $\Id-\cV_1\cF=\Id-\cC_S\cF$ and $\norm{Q}\norm{Q^{-1}}\lesssim 1$ (cf.~\eqref{eq:q_sim_1}), 
the operator $\BO_0$ inherits the properties listed in the lemma from $\Id-\cC_S\cF$. Finally, we can perform analytic perturbation theory for the simple isolated eigenvalue in $D_\eps(0)$ of $\BO_0$ to verify the lemma for $\TO=\BO=\BO_0 + \cE$ with 
$\cE=\ord(\dens)$ (cf.~\eqref{eq:E=ord rho}) and $\TO = \BO_* = \BO_0 + \cE_*$ with $\cE_* = \ord(\dens)$ (cf.~\eqref{eq:E_star_ord_rho}) if $\dens_*$ is sufficiently small. 
This completes the proof of Lemma~\ref{lem:prop_F_small_dens}. 
\end{proof}

In the following corollary, we use the concepts of \emph{left} and \emph{right eigenvector} of an operator 
$\TO \colon \C^{N\times N} \to \C^{N\times N}$. We say $V_l \in \C^{N\times N}$ ($V_r\in \C^{N\times N}$) is a left (right) 
eigenvector of $\TO$ corresponding to the eigenvalue $\lambda \in \C$ of $\TO$ if 
$\TO^*[V_l] = \bar \lambda V_l$ ($\TO[V_r] = \lambda V_r$).

\begin{corollary} \label{coro:eigenvector_expansion}
Let $z \in \Dbdd$ satisfy $\dens(z) + \eta \dens(z)^{-1}\leq \dens_*$ for $\dens_* \sim 1$ from Lemma \ref{lem:prop_F_small_dens}. 

Let $\beta_0$ and $\beta$ be the isolated eigenvalues in $D_\eps(0)$ of $\BO_0$ and $\BO$, respectively, from Lemma~\ref{lem:prop_F_small_dens}. 
Furthermore, let $\cP_0=\cP_{\BO_0}$ and $\cP=\cP_{\BO}$ be the spectral projections corresponding to the isolated eigenvalue of $\BO_0$ and $\BO$, respectively (see \eqref{eq:nondegeneracy for CF}).
Then with $\cQ_0\defeq \Id -\cP_0$ and $\cQ\defeq \Id-\cP$ we have 
\begin{equation} \label{eq:B_inverse_Q_norm}
\norm{\BO^{-1} \cQ}  + \normsp{\BO^{-1}\cQ}  +\norm{\BO_0^{-1} \cQ_0}  \lesssim 1.
\end{equation}
We define $B_0 \defeq \cP_0\cC_{Q^*,Q}[F_U]$ and $P_0 \defeq \cP^*_0\cC_{Q,Q^*}^{-1}[F_U]$. Then $B_0$ and $P_0$ are right and left eigenvector of $\BO_0$ corresponding to $\beta_0$
and we have 
\begin{subequations} 
\begin{align} 
 B_0 & = \cC_{Q^*,Q}[F_U] + \ord(\eta \dens^{-1}), \qquad \qquad P_0 = \cC_{Q,Q^*}^{-1} [F_U] + \ord(\eta\dens^{-1}), \label{eq:b_0_l_0_approx}\\ 
\beta_0 &=\frac{\eta}{\dens} \frac{\pi}{\avg{F_U^2}} +\ord(\eta^2\dens^{-2}) = \ord(\eta\dens^{-1})\,.& \label{eq:beta_0_approx}
\end{align}
\end{subequations}
We also define $B \defeq \cP[B_0]$ and $P \defeq \cP^*[P_0]$. This yields right and left eigenvectors of $\BO$ corresponding to $\beta$ which satisfy  
\begin{subequations}
\label{eq:expansion of beta b l}
\begin{align}
B \,&=\, B_0 + \ord(\dens)\,, \label{eq:expansion_b}
\\
P \,&=\, P_0 + \ord(\dens)\,, \label{eq:expansion_l}
\\
\beta\scalar{P}{B}\,&=\, \pi \eta\dens^{-1} - 2\ii\dens\sigma + \ord( \dens^2 + \eta + \eta^2\dens^{-2})\,. 
\label{eq:expansion_beta_scalar_l_b} 
\end{align}
\end{subequations}
Moreover, we have 
\begin{equation} \label{eq:b_l_bounded} 
\norm{B} \lesssim 1, \qquad \qquad \norm{P} \lesssim 1, \qquad \qquad \abs{\scalar{P}{B}} \sim 1.
\end{equation}
\end{corollary}

The following identity will be used a few times 
\begin{equation} \label{eq:f_u_qq_star}
 \avg{F_UQQ^*} = \dens^{-1}\avg{\Im M}=\pi.
 \end{equation}
It is obtained by a direct computation starting from the definition of $F_U$ in \eqref{eq:def_S_F_U_Sigma}, the balanced polar decomposition, $M=Q^*UQ$, and $\dens(z) = \pi^{-1} \avg{\Im M(z)}$.

\begin{proof}

The bounds in \eqref{eq:B_inverse_Q_norm} are a direct consequence of Lemma~\ref{lem:prop_F_small_dens}.  
Using \eqref{eq:F_f_u} and $\cC_S[F_U] = F_U$, we see that 
\begin{equation} \label{eq:B_0C_q^*q_f_u} 
\BO_0^*\cC_{Q,Q^*}^{-1}[F_U]\,=\, \eta \dens^{-1}\id \,, \qquad \BO_0\cC_{Q^*,Q}[F_U]\,=\, \ord(\eta\dens^{-1})\,. 
\end{equation}
The representations of $B_0$ and $P_0$  in \eqref{eq:b_0_l_0_approx} follow by simple perturbation theory because $\beta_0$ is a nondegenerate isolated eigenvalue.
The expression for $\beta_0$ in \eqref{eq:beta_0_approx} is seen by taking the scalar product with $B_0$ in the first identity of \eqref{eq:B_0C_q^*q_f_u} as well as using \eqref{eq:b_0_l_0_approx} and \eqref{eq:f_u_qq_star}. 
  
 The expansions \eqref{eq:expansion of beta b l} follow by first order analytic perturbation theory. Indeed, $B = B_0 + \ord(\dens)$ and $P=P_0 + \ord(\dens)$ as $\cE = \BO-\BO_0 = \ord(\dens)$ due to \eqref{eq:E=ord rho}. 
For the proof of \eqref{eq:expansion_beta_scalar_l_b}, we first compute $\cE[B_0]$. 
From \eqref{eq:b_0_l_0_approx}, we obtain the first equality below: 
\begin{equation} \label{eq:E_b_0} 
 \cE[B_0]  = \cC_{Q^*,Q} (\cC_S - \cC_U) \cF[F_U] + \ord(\eta) = -2 \ii \dens \cC_{Q^*,Q} [SF_U^2] +  \ord(\dens^2+ \eta), 
\end{equation}
For the second equality in \eqref{eq:E_b_0}, we used \eqref{eq:F_f_u}, $\norm{\cC_S - \cC_U}= \ord(\dens)$ and $ (\cC_S - \cC_U)[F_U] = 2( \Im U -\ii \Re U )(\Im U)F_U = -2 \ii \dens S F_U^2  +\ord(\dens^2)$ 
due to \eqref{eq:Re_u_s}.  
For the proof of \eqref{eq:expansion_beta_scalar_l_b}, we start from 
$\BO[B] = \beta B$, $\BO = \BO_0 + \cE$, use \eqref{eq:expansion_b}, \eqref{eq:expansion_l} as well as $\cE = \ord(\dens)$ and obtain 
\begin{equation} \label{eq:beta_scalar_l_b_lemma}
 \beta \scalar{P}{B} = \beta_0 \scalar{P_0}{B_0} + \scalar{P_0}{\cE[B_0]} + \ord(\dens^2).  
\end{equation}
Together with the following two expansions, this yields \eqref{eq:expansion_beta_scalar_l_b}. 
We have 
\begin{align*}
 \beta_0 \scalar{P_0}{B_0} & = \pi \eta \dens^{-1} + \ord(\eta^2\dens^{-2}), \\ 
\scalar{P_0}{\cE[B_0]} & = - 2\ii \dens\avg{SF_U^3} + \ord(\dens^2 + \eta) = -2 \ii \dens \sigma + \ord(\dens^2 + \eta). 
\end{align*}
The first expansion is a consequence of $\scalar{P_0}{B_0} = \avg{F_U^2} + \ord(\eta \dens^{-1})$ due to \eqref{eq:b_0_l_0_approx} and \eqref{eq:beta_0_approx}.  
The second expansion follows from \eqref{eq:b_0_l_0_approx} and \eqref{eq:E_b_0}.

The first two bounds in \eqref{eq:b_l_bounded} follow directly from \eqref{eq:expansion_b} and \eqref{eq:expansion_l} as well as \eqref{eq:b_0_l_0_approx}, \eqref{eq:q_sim_1} and \eqref{eq:im_u_sim_rho}. 
Moreover, \eqref{eq:b_0_l_0_approx}, \eqref{eq:expansion_b} and \eqref{eq:expansion_l} imply $\abs{\scalar{P}{B}} \sim \avg{F_U^2} \sim 1$ by \eqref{eq:im_u_sim_rho}.
 This completes the proof of Corollary~\ref{coro:eigenvector_expansion}.  
\end{proof}

\begin{proof}[Proof of Proposition~\ref{pro:B_inverse_improved_bound}]
As in the proof of Lemma~\ref{lem:prop_F_small_dens}, it suffices to show the bound on $\normsp{\BO^{-1}}$ in \eqref{eq:B_inverse_improved_bounds}. 

From \eqref{eq:representation_BO}, by using Lemma~\ref{lem:prop_m_q}, we conclude that 
\[ \normsp{\BO^{-1}} \lesssim \normsp{(\cC_U^* - \cF)^{-1}} \lesssim \abs{1- \normsp{\cF} \scalar{F}{C_U^*[F]}}^{-1} \lesssim \Big(1- \normsp{\cF} + \abs{1- \scalar{F}{\cC_U^*[F]}}\Big)^{-1}. \] 
Here, we applied the Rotation-Inversion Lemma, Lemma 4.9 in \cite{2016arXiv160408188A}, with $\TO=\cF$ and $\mathcal{U} = \cC_U^*$ in the second step. Its conditions are met due to Fact \eqref{item:F_simple_eigenvalue} and Fact \eqref{item:F_spectral_gap} about $\cF$ from the proof of Lemma~\ref{lem:prop_F_small_dens}. 

Owing to \eqref{eq:normtwo_F} as well as \eqref{eq:q_sim_1} and \eqref{eq:im_u_sim_rho}, we have $1 - \normsp{\cF} \sim \eta\dens^{-1}$. 
Therefore, it suffices to show that 
\begin{equation} \label{eq:proof_bound_inverse_BO_aux}
 \abs{1- \scalar{F}{\cC_U^*[F]}} \gtrsim \dens(\dens + \abs{\sigma}) 
\end{equation}
when $\eta \dens^{-1}$ is small. 
As $1 \geq \avg{F \Re U F \Re U}$ due to $\normtwo{F} = 1$, we estimate 
\[ \abs{1- \scalar{F}{\cC_U^*[F]}} = \abs{1- \avg{FU^* FU^*}} \gtrsim \avg{F \Im U F \Im U} + \abs{\avg{F\Im U F \Re U}}. \] 
Since $\Im U \sim \dens$ by \eqref{eq:im_u_sim_rho}, the first term on the right-hand side scales like $\sim \dens^2$. This proves \eqref{eq:proof_bound_inverse_BO_aux} 
when $\dens \geq \dens_*$ for any $\dens_* \sim 1$ as $\abs{\sigma} \lesssim 1$. If $\dens_*$ is sufficiently small and $\dens + \eta\dens^{-1} \leq \dens_*$ then we use 
$\avg{F\Im U F \Re U } = \dens\normtwo{F_U}^{-2} \avg{F_U^3S} + \ord(\dens^{3} + \eta)$ by \eqref{eq:f_u_approx_f} and \eqref{eq:Re_u_s}
to conclude \eqref{eq:proof_bound_inverse_BO_aux} and, thus, \eqref{eq:B_inverse_improved_bounds} in the missing regime. 
This completes the proof of Proposition~\ref{pro:B_inverse_improved_bound}. 
\end{proof}

\subsection{Sharp bound on \texorpdfstring{$\BO^{-1}$}{B-1} and \texorpdfstring{$1/2$}{1/2}-Hölder continuity of \texorpdfstring{$M$}{M}} 
\label{subsec:proof_hoelder_continuity} 

In this section, we will prove Theorem~\ref{thm:B_inverse_sharp_edge} and Corollary~\ref{cor:Hoelder_continuity_M}. 
They will be proven directly after the following proposition, the main result of the present section.
It shows that $\sigma$ introduced 
in \eqref{eq:def_S_F_U_Sigma} is of order one close to regular edges $\tau_0\in\partial\supp\varrho$. 
For the formulation of this proposition, we define 
\begin{equation} \label{eq:def_A}
 \cA[R,T] \defeq \frac{1}{2} \Big( M\SS[R]T + T \SS[R] M \Big)  
\end{equation}
with $R, T \in \C^{N\times N}$. 

\begin{proposition} \label{pro:sigma_sim_1}
Let \ref{assumption A}, \ref{assumption flatness} and \ref{assumption bounded M} be satisfied for some $\tau_0 \in \R$. If $\tau_0\in\partial\supp\varrho$ is a regular edge then the following statements hold true
\begin{enumerate}[(i)]
\item At $z = \tau_0$, for $P$ and $B$ defined as in Corollary~\ref{coro:eigenvector_expansion}, we have 
\[ \abs{\scalar{P}{\cA[B,B]}} \sim 1.\] 
\item There is $\delta_* \sim 1$ such that 
\[ \abs{\sigma(z)} \sim 1\] 
for all $z \in \Hb$ satisfying $\abs{z-\tau_0} \leq \delta_*$. 
\end{enumerate}
\end{proposition} 

Proposition~\ref{pro:sigma_sim_1} immediately implies Theorem~\ref{thm:B_inverse_sharp_edge} and Corollary~\ref{cor:Hoelder_continuity_M}. 

\begin{proof}[Proof of Theorem~\ref{thm:B_inverse_sharp_edge}]
By Proposition~\ref{pro:sigma_sim_1} (ii), there is $\delta_* \sim 1$ such that $\abs{\sigma(z)} \sim 1$ for all $z \in \Hb$ satisfying $\abs{z- \tau_0} \leq \delta_*$. 
Therefore, Theorem~\ref{thm:B_inverse_sharp_edge} follows directly from Proposition~\ref{pro:B_inverse_improved_bound}. 
\end{proof} 

\begin{proof}[Proof of Corollary~\ref{cor:Hoelder_continuity_M}] 
We proceed exactly as in the proof of Corollary~\ref{cor:hoelder_1_3} but use Theorem~\ref{thm:B_inverse_sharp_edge} instead of \eqref{eq:B_inverse_improved_bounds} 
for all $z \in \Hb$ such that $\abs{z- \tau_0} \leq \delta_*$, where $\delta_*$ is chosen as in Theorem~\ref{thm:B_inverse_sharp_edge}.
\end{proof} 

The proof of Proposition~\ref{pro:sigma_sim_1} requires two auxiliary lemmas whose proofs are postponed until the end of the section. 
Some statements in these lemmas will be stated for more general $\tau_0\in \R$
not only when $\tau_0$ is a regular edge, although we will eventually use them in this case. 

We now choose $\theta = \omega_*/2$ in Corollary \ref{cor:hoelder_1_3} and work on the set $\Dtheta$ in the following.  
Note that $\Dtheta \subset \Dbdd$. 
By Hölder-continuity we can then extend $M$ to $\overline{\Dtheta}$, and we denote the extension by $M$ as well. Moreover, the operators $\BO$ and $\BO_*$ are defined for all $z \in \overline{\Dtheta}$ and the results about $\BO$ and $\BO_*$ in 
Lemma~\ref{lem:prop_F_small_dens} hold true on $\overline{\{z \in \Dtheta \colon \dens(z) + \eta \dens(z)^{-1}\leq \dens_*\}}$, 
where the closure is taken with respect to the Euclidean topology on $\C$. 
Lemma~\ref{lem:area_small_density} below shows that this set contains a neighbourhood 
around any point $\tau_0 \in \pt\supp \dens$.

\begin{lemma} \label{lem:area_small_density} 
Let \ref{assumption A}, \ref{assumption flatness} and \ref{assumption bounded M} hold true for some $\tau_0 \in \R$. 
Then the following holds true:
\begin{enumerate}[(i)]
\item 
There is $\dens_* \sim 1$ such that, for the eigenvalue $\beta_*$ of $\BO_*=\Id-\cC_{M^*,M}\SS$ in $D_\eps(0)$ (cf.~Lemma~\ref{lem:prop_F_small_dens}), we have   
\begin{equation} \label{eq:abs_beta_star_sim_eta_rho}
 \abs{\beta_*} \sim \eta /\dens 
\end{equation}
uniformly for $z \in \Dbdd$ satisfying $\dens(z) + \eta \dens(z)^{-1} \leq \dens_*$. 
\item If $\tau_0 \in \pt\supp\dens$ and $\dens_* \sim 1$ then there is $\delta_* \sim 1$ such that $\dens(z) + \eta \dens(z)^{-1} \leq \dens_*$ for all $z \in \Hb$ satisfying $\abs{z- \tau_0} \leq\delta_*$. 

Moreover, we have 
\begin{equation} \label{eq:lim_eta_dens_inverse}
 \lim_{\eta \downarrow 0} \eta \dens(\tau_0 + \ii \eta)^{-1} = 0. 
\end{equation}

\end{enumerate} 
\end{lemma} 

\begin{lemma} \label{lem:coefficient_quadratic_term}
Let \ref{assumption A}, \ref{assumption flatness} and \ref{assumption bounded M} be satisfied for some $\tau_0 \in \R$. 
Then there is $\dens_* \sim 1$ such that, uniformly for all $z \in \Dbdd$ satisfying $\dens(z) + \eta\dens(z)^{-1}\leq \dens_*$, we have 
\begin{subequations} 
\begin{align} 
\scalar{P}{\cA[B,B]} & = \sigma +\ord(\dens +  \eta \dens^{-1}), \label{eq:scalar_l_m_Sb_b_b_Sb_m}\\ 
 \scalar{P}{M\SS[B]B} & = \sigma + \ord(\dens + \eta \dens^{-1}). \label{eq:scalar_l_m_Sb_b} 
\end{align} 
\end{subequations} 
\end{lemma} 

We remark that \eqref{eq:scalar_l_m_Sb_b} will be used in the next section. 

\begin{proof}[Proof of Proposition~\ref{pro:sigma_sim_1}] 
In this proof, we will analyse $M$ and $\BO=\Id - \cC_M \SS$ on the real line outside the self-consistent spectrum, i.e.~we will consider spectral parameters $z = \tau + \ii \eta$ such that 
$ \tau \in [\tau_0 -\omega_*/2, \tau_0 + \omega_*/2] \setminus \supp \dens$ and $\eta =0$.  
In particular, $\dens(\tau) = 0$ and thus $M=M^*$ by \eqref{eq:im_M_sim_avg}. 
Owing to the continuity of $M$ (Corollary~\ref{cor:hoelder_1_3}), $M$ satisfies the MDE, \eqref{MDE}, also for these spectral parameter $z$. Moreover, 
$\dens(\tau + \ii \eta) \lesssim \eta/ \dist(\tau + \ii \eta, \supp \dens)^2$ as $\avg{M}$ is the Stieltjes transform of the measure $\mu$ on $\R$ (compare \eqref{eq:Stieltjes_representation}). 
Thus, $\BO$ is invertible at $\tau \notin \supp \dens$ due to Proposition~\ref{pro:B_inverse_improved_bound} as the term $\eta \dens^{-1}$ has a uniform lower bound for $z = \tau + \ii \eta$ with $\eta>0$. 
In particular, $M$ and $\beta$ are differentiable with respect to $\omega = \tau- \tau_0$ for $\tau \notin \supp \dens$. First order perturbation theory of the isolated eigenvalue $\beta$ of the non-selfadjoint operator $\BO$ yields 
\begin{equation} \label{derivative of lambda}
\pt_\omega \beta = -\frac{\scalar{P}{\cC_{\pt_\omega M,M}\SS[B]}}{\scalar{P}{B}}- \frac{\scalar{P}{\cC_{M,\pt_\omega{M}}\SS[B]}}{\scalar{P}{B}}
= - \frac{\scalar{P}{(\pt_\omega M)\SS[B] M + M \SS[B](\pt_\omega M)}}{\scalar{P}{B}}. 
\end{equation}
For definiteness, we assume in the following that $\tau_0$ is a right edge. Hence, $\omega>0$. The argument for a left 
edge works completely analogously. 

Owing to the invertibility of $\BO$, the MDE, \eqref{MDE}, is differentiable at $\tau$ with respect to $\omega$. 
Similarly to \eqref{basic MDE equations}, we obtain  
\[ \pt_\omega M = \BO^{-1}[M^2] = \frac{\scalar{P}{M^2}}{\beta\scalar{P}{B}} B + \BO^{-1}\cQ[M^2]. \] 
In the second step, we inserted $\cP + \cQ = \Id$ and employed the definition of $\cP=\cP_{\BO}$ in 
Corollary~\ref{coro:eigenvector_expansion}. 
We insert this into \eqref{derivative of lambda} and get from Lemma~\ref{lem:prop_F_small_dens} and \eqref{eq:B_inverse_Q_norm} that 
\[  \pt_\omega \beta = -\frac{\scalar{P}{M^2}}{\beta \scalar{P}{B}^2} \scalar{P}{B\SS[B]M + M \SS[B] B} + \ord(1) = \frac{2\scalar{P}{M^2}}{\beta \scalar{P}{B}^2} \scalar{P}{\cA[B,B]} + \ord(1). \] 
The bounds in \eqref{eq:b_l_bounded} of Corollary~\ref{coro:eigenvector_expansion} yield  $\norm{P} \lesssim 1$ and, hence, $\abs{\scalar{P}{M^2}} \lesssim 1$ by Assumption~\ref{assumption bounded M}. 
By \eqref{eq:b_l_bounded}, we have $\abs{\scalar{P}{B}}\sim 1$ if $\eta >0$. 
 Thus, as a consequence of the continuity of $M$ by Corollary~\ref{cor:hoelder_1_3} and, hence, of $P$ and $B$, 
the derivative of $\beta^2$ is bounded by  
$\abs{\pt_\omega (\beta^2) } \lesssim \abs{\scalar{P}{\cA[B,B]}} + \abs{\beta}$.
This implies 
\begin{equation} \label{eq:abs_beta_squared} 
\abs{\beta}^2 \lesssim \abs{\scalar{P}{\cA[B,B]}}\omega + \omega^2. 
\end{equation}
On the other hand, from \eqref{eq:abs_beta_star_sim_eta_rho} and the continuity of $\beta_*$, and $\beta_*=\beta$ for $\eta=0$ (as $M=M^*$) we get 
\[
\abs{\beta(\tau_0+\omega)} \sim \lim_{\eta \downarrow 0} \frac{\eta}{\dens(\tau_0+\omega + \ii \eta)} \sim  \bigg(\int_0^\delta \frac{\dens(\tau_0 -\omega')}{(\omega' +\omega)^2}\di \omega'\bigg)^{-1},
\]
for some $\delta \sim 1$. From this and \eqref{sqrt growth}, we conclude that 
\[
\liminf_{\omega \downarrow 0} \frac{\abs{\beta(\tau_0 + \omega)}}{\sqrt{\omega}}\sim \limsup_{\omega \downarrow 0} \frac{\abs{\beta(\tau_0 + \omega)}}{\sqrt{\omega}}\sim 1\,,
\]
i.e.~$\abs{\beta}^2 \sim \omega$ as $\omega \downarrow 0$. 
Therefore, we find $\abs{\scalar{P}{\cA[B,B]}} \gtrsim 1$ at $z=\tau_0$ due to \eqref{eq:abs_beta_squared}. The upper bound follows from $\norm{P} \lesssim 1$ and $\norm{B} \lesssim 1$ by Corollary~\ref{coro:eigenvector_expansion}. This completes the proof of (i). 

For the proof of (ii), we conclude that $\scalar{P}{\cA[B,B]}$ is a uniformly $1/3$-Hölder continuous function of $z$ on $\{w \in \Hb \cup \R \colon \abs{w - \tau_0} \leq \delta_*\}$ 
for some $\delta_* \sim 1$ due to Corollary~\ref{cor:hoelder_1_3} and Lemma~\ref{lem:area_small_density} (ii). By possibly shrinking $\delta_* \sim 1$, we can thus assume that 
$\abs{\scalar{P}{\cA[B,B]}} \sim 1$ for all $z \in \Hb$ satisfying $\abs{z-\tau_0} \leq \delta_*$. From Lemma~\ref{lem:area_small_density} (ii) and \eqref{eq:scalar_l_m_Sb_b_b_Sb_m}, 
we conclude that $\abs{\sigma(z)} \sim 1$ for all $z \in \Hb$ such that $\abs{z-\tau_0} \leq \delta_*$ for some sufficiently small $\delta_* \sim 1$. 
Hence, we have completed the proof of Proposition~\ref{pro:sigma_sim_1}. 
\end{proof} 

\begin{proof}[Proof of Lemma~\ref{lem:area_small_density}] 
Similarly to the proof of Corollary~\ref{coro:eigenvector_expansion}, we find a left eigenvector $P_*$ of $\BO_*$ corresponding to $\beta_*$, i.e.~
$(\BO_*)^*[P_*] = \overline{\beta_*}P_*$, such that 
\begin{equation} \label{eq:expansion_P_star}
 P_* = Q^{-1}F_U(Q^*)^{-1} + \ord(\dens + \eta\dens^{-1}) 
\end{equation}
provided that $z \in \Dbdd$ satisfies $\dens(z) + \eta \dens(z)^{-1} \leq \dens_*$. 
We take the imaginary part of \eqref{MDE} and compute the scalar product with $P_*$. 
This yields 
\begin{equation} \label{eq:eigenvalue_formula_B_star}
 \beta_* = \frac{\eta}{\dens}\frac{ \scalar{P_*}{M^*M}}{\scalar{P_*}{\dens^{-1}\Im M}}. 
\end{equation}
Using \eqref{eq:expansion_P_star} and the balanced polar decomposition, $M = Q^* U Q$, we obtain
\[ \begin{aligned} 
\scalar{P_*}{M^*M} & = 
\avg{F_U QQ^*} + \ord(\dens + \eta\dens^{-1})  = \pi+ \ord(\dens + \eta \dens^{-1}), \\ 
\scalar{P_*}{\dens^{-1}\Im M} & = \avg{F_U^2} + \ord(\dens + \eta \dens^{-1}).  
\end{aligned}  \] 
Here, we used that $U$ and $F_U$ commute and \eqref{eq:f_u_qq_star} in order to compute $\scalar{P_*}{M^*M}$. 
We thus deduce that $\abs{\scalar{P_*}{M^*M}} \sim 1$ and $\abs{\scalar{P_*}{\dens^{-1} \Im M}} \sim 1$ for all $z \in \Dbdd$ satisfying $\dens(z) + \eta \dens(z)^{-1} \leq \dens_*$ 
for some sufficiently small $\dens_* \sim 1$ due to \eqref{eq:im_u_sim_rho}. 
Therefore, taking the absolute value in \eqref{eq:eigenvalue_formula_B_star} and using these scaling relations complete the proof of \eqref{eq:abs_beta_star_sim_eta_rho}. 

For the proof of (ii), we remark that, owing to the continuity of $\dens$, we have  
\[ \lim_{\eta \downarrow 0} \dens(\tau + \ii \eta)^{-1} \eta = 0 \] 
for all $\tau \in \R$ satisfying $\dens(\tau) >0$. From \eqref{eq:abs_beta_star_sim_eta_rho}, we thus conclude that 
$\beta_*(\tau) = 0$ if $\dens(\tau)>0$ for all $\tau \in [\tau_0-\omega_*/2, \tau_0 + \omega_*/2]$. 
The continuity of $M$ from Corollary~\ref{cor:hoelder_1_3} implies that $\BO_*$ is also $1/3$-Hölder continuous. 
Consequently, $\beta_*$ is also $1/3$-Hölder continuous as it is an isolated eigenvalue of $\BO_*$. 
Owing to the continuity of $\dens$, we find a sequence $(\tau_n)_n$ such that $\tau_n \to \tau_0 \in \pt\supp \dens$ 
and $\dens(\tau_n) >0$ for all $n$. 
Thus, the continuity of $\beta_*$ yields $\beta_*(\tau_0) = 0$. Therefore, we have $\abs{\beta_*} + \dens = 0$ at $z = \tau_0$. Hence, the $1/3$-Hölder continuity of $\abs{\beta_*} + \dens$ implies that there is $\delta_* \sim 1$ 
such that $\dens(z) + \eta \dens(z)^{-1} \leq \dens_*$ since $\dens + \eta \dens^{-1} \sim \dens + \abs{\beta_*}$ by \eqref{eq:abs_beta_star_sim_eta_rho}. 
From $\beta_*(\tau_0) =0$ and \eqref{eq:abs_beta_star_sim_eta_rho}, we directly conclude \eqref{eq:lim_eta_dens_inverse}.  
This completes the proof of Lemma~\ref{lem:area_small_density}.
\end{proof}

\begin{proof}[Proof of Lemma~\ref{lem:coefficient_quadratic_term}]
First, we use the balanced polar decomposition, $M=Q^* U Q$, \eqref{eq:def_F} and the definition of $\cA$ in \eqref{eq:def_A} to obtain 
\begin{equation} \label{eq:A_x_y_alternative}
\cA[R,T] = \frac{1}{2} \cC_{Q^*,Q}\Big[ U (\cF \cC_{Q^*,Q}^{-1}[R])\cC_{Q^*,Q}^{-1} [T]+ \cC_{Q^*,Q}^{-1}[T] (\cF\cC_{Q^*,Q}^{-1}[R])U \Big] 
\end{equation}
for $R, T \in \C^{N\times N}$. 

We choose $\dens_* \sim 1$ small enough such that Lemma \ref{lem:prop_F_small_dens} is applicable. 
By using $U = S + \ord(\dens)$ due to \eqref{eq:Re_u_s} as well as \eqref{eq:def_S_F_U_Sigma}, \eqref{eq:F_f_u} and \eqref{eq:b_0_l_0_approx} 
in \eqref{eq:A_x_y_alternative}, we get
\begin{equation} \label{eq:A_b_0_b_0}
 \cA[B_0,B_0] = \cC_{Q^*,Q}[S F_U^2] + \ord( \dens + \eta\dens^{-1}). 
\end{equation}
In order to show \eqref{eq:scalar_l_m_Sb_b_b_Sb_m}, we use \eqref{eq:expansion_b} as well as \eqref{eq:expansion_l} and obtain 
 \[ \scalar{P}{\cA[B,B]} = \,  \scalar{P_0}{\cA[B_0,B_0]} + \ord(\dens) 
= \,   \avg{SF_U^3} + \ord(\dens+\eta\dens^{-1})  
= \,  \sigma + \ord(\dens+ \eta \dens^{-1}).\] 
This completes the proof of \eqref{eq:scalar_l_m_Sb_b_b_Sb_m}. A similar computation yields \eqref{eq:scalar_l_m_Sb_b}. 
\end{proof} 

\subsection{Derivation of the quadratic equation}
In this section, we expand $M(\tau_0 + \omega)$ around $M(\tau_0)$ for a regular edge $\tau_0 \in \pt\supp\dens$. We show that this approximation is to leading order dominated by a 
scalar-valued quantity, $\Theta$, which satisfies a quadratic equation. That is the content of the following 
proposition which is the main result of this section.

\begin{proposition}[Quadratic equation for shape analysis] \label{pro:quadratic_for_shape_analysis}
Let \ref{assumption A}, \ref{assumption flatness} as well as \ref{assumption bounded M} be satisfied for some regular edge $\tau_0\in\partial\supp\varrho$. 
Then there is $\delta_* \sim 1$ such that the following hold true:
\begin{enumerate}[(a)] 
\item 
For all $\omega\in [-\delta_*, \delta_*]$, we have   
\begin{equation} \label{eq:decomposition_diff_M}
 M(\tau_0 + \omega) -M(\tau_0) = \Theta(\omega) B + R(\omega), 
\end{equation}
where $\Theta \colon [-\delta_*, \delta_* ] \to \C$ and $R \colon [-\delta_*, \delta_*] \to \C^{N\times N}$ are defined by 
\begin{equation} \label{eq:decomposition_m_Theta_r} 
 \Theta(\omega) \defeq \scalarbb{\frac{P}{\scalar{B}{P}}}{M(\tau_0 +\omega) - M(\tau_0)}, \qquad R(\omega) \defeq \cQ[M(\tau_0 +\omega) - M(\tau_0)]. 
\end{equation}
Here, $P=P(\tau_0)$, $B=B(\tau_0)$ and $\cQ=\cQ(\tau_0)$ are the eigenvectors and spectral projection of $\BO(\tau_0)$ introduced in Corollary~\ref{coro:eigenvector_expansion}.
We have $B = B^*$ and $P = P^*$ as well as $B\sim 1$ and $P \sim 1$. 
Moreover, $\Theta(\omega)$ and $R(\omega)$ are bounded by 
\begin{equation}\label{eq:bound_Theta_R} 
 \abs{\Theta(\omega)} \lesssim \abs{\omega}^{1/2},\qquad \Im \Theta(\omega) \geq 0, \qquad \norm{\Im R(\omega)} \lesssim \abs{\omega}^{1/2}\Im \Theta(\omega) 
\end{equation}
uniformly for all $\omega \in [-\delta_*, \delta_*] $. 
\item \label{item:quadratic_equation} The function $\Theta$ satisfies the quadratic equation 
\begin{equation} \label{eq:quadratic_equation}
\sigma \Theta^2(\omega) + \omega \Xi(\omega) = 0, \qquad \qquad \Xi(\omega)  = \pi(1 + \nu(\omega)), 
\end{equation}
for all $\omega \in [-\delta_*, \delta_*]$, where $\sigma = \scalar{P}{M\SS[B]B}$, $M=M(\tau_0)$, and the error term $\nu(\omega)$ satisfies
\begin{equation}\label{eq:bound_nu}
\abs{\nu(\omega)} \lesssim \abs{\omega}^{1/2}, \qquad \abs{\Im \nu(\omega)} \lesssim \Im \Theta(\omega) 
\end{equation}
for all $\omega \in [-\delta_*, \delta_* ]$. 
\end{enumerate}
\end{proposition}

The definition $\sigma = \scalar{P}{M\SS[B]B}$ for $\tau_0 \in \pt \supp \dens$ extends  
the definition of $\sigma$ in \eqref{eq:def_S_F_U_Sigma} on $\Hb$ owing to \eqref{eq:scalar_l_m_Sb_b}, 
\eqref{eq:lim_eta_dens_inverse} as well as  the continuity of $M$ and, thus, $P$, $B$ and $\dens$. 

We warn the reader that, in this section, 
functions of $z$ like $M$, $B$, $P$, $U$, $Q$, etc.~without argument are understood to be evaluated 
at $\tau_0$ instead of the generic spectral parameter $z$ which is the convention in most of the other parts of this work.

\begin{proof} 
The first bound in \eqref{eq:bound_Theta_R} follows directly from Corollary~\ref{cor:Hoelder_continuity_M}. 

From \eqref{eq:b_0_l_0_approx}, \eqref{eq:expansion_b}, \eqref{eq:expansion_l}, $\dens(\tau_0)=0$ and 
\eqref{eq:lim_eta_dens_inverse}, we conclude that $B$ and $P$ are the limits of Hermitian, positive-definite
matrices which are $\sim 1$ due to Lemma~\ref{lem:prop_m_q}. Thus, $B=B^* \sim 1$ and $P=P^* \sim 1$. 
This also implies that $\Im \Theta(\omega) \geq 0$ in \eqref{eq:bound_Theta_R} as $\Im M(\tau_0 + \omega)$ 
is always positive semidefinite and $\Im M(\tau_0) = 0$. 
 
In the following lemma whose proof we postpone till the end of this section 
we establish a quadratic equation for $\Theta$. 

\begin{lemma}[Derivation of the quadratic equation]  \label{lem:derivation_quadratic_equation}
Let $\Theta(\omega)$ and $R(\omega)$ be defined as in \eqref{eq:decomposition_m_Theta_r} and $\cA$ be defined 
as in \eqref{eq:def_A}. Then there is $\delta_* \sim 1$ such that, for all $\omega \in [-\delta_*,\delta_*]$,
 $\Theta=\Theta(\omega)$ satisfies the quadratic equation 
\[ \mu_2 \Theta^2 + \mu_1 \Theta + \mu_0 = e(\omega) \] 
with some error term $ e(\omega) = \ord(\abs{\omega}^{3/2})$ and with coefficients 
\begin{equation}  \label{eq:def_coefficients_quadratic_equation}
\mu_2  = \scalar{P}{\cA[B,B]}, \qquad  \mu_1  = - \beta\scalar{P}{B}, \qquad  \mu_0  = \omega \scalar{P}{M^2}. 
\end{equation}
Moreover, for all $\omega \in [-\delta_*, \delta_*]$, we have 
\begin{equation} \label{eq:bounds_imaginary_part}
 \abs{\Im e(\omega)} \lesssim  \abs{\omega} \Im \Theta(\omega), \qquad\qquad \norm{\Im R(\omega)} \lesssim \abs{\omega}^{1/2}\Im \Theta(\omega). 
\end{equation}
\end{lemma} 

We now compute the coefficients defined in \eqref{eq:def_coefficients_quadratic_equation} precisely. This will 
yield the quadratic equation in \eqref{eq:quadratic_equation}.  

Owing to \eqref{eq:scalar_l_m_Sb_b_b_Sb_m}, \eqref{eq:scalar_l_m_Sb_b}, \eqref{eq:lim_eta_dens_inverse}, $\dens(\tau_0) =0$ and the continuity of $M$ and, thus, $P$, $B$ and $\dens$,
we have $\mu_2 = \sigma$ as defined in Proposition~\ref{pro:quadratic_for_shape_analysis} \eqref{item:quadratic_equation}.

The expansion in \eqref{eq:expansion_beta_scalar_l_b} implies $\mu_1 = 0$ at $\tau_0$ by \eqref{eq:lim_eta_dens_inverse}. We now compute $\mu_0$. 
At $z \in\Hb$ satisfying $\dens(z) + \dens(z)^{-1} \Im z \leq \dens_*$, we conclude from \eqref{eq:expansion_l}, 
\eqref{eq:b_0_l_0_approx} 
and the balanced polar decomposition, $M=Q^* U Q$, from \eqref{eq:balanced_polar_decomposition} that  
\[ \scalar{P}{M^2} = \scalar{Q^{-1}F_U(Q^*)^{-1}}{Q^* U QQ^* U Q} + \ord(\dens+\eta \dens^{-1}) = 
\avg{F_UQQ^*} +\ord(\dens+ \eta\dens^{-1}) = \pi  + \ord(\dens+ \eta\dens^{-1}). \] 
Here, we also employed that $U = S + \ord(\dens)$ by \eqref{eq:Re_u_s} and $F_U$ and $S$ commute in the 
second step and \eqref{eq:f_u_qq_star} in the last step. 
Thus, we have $\mu_0 = \omega \pi$ at $\tau_0$ by \eqref{eq:lim_eta_dens_inverse}. 

We set $\nu(\omega) \defeq -(\pi \omega)^{-1} e(\omega)$ with $e(\omega)$ as introduced in Lemma~\ref{lem:derivation_quadratic_equation}. 
This immediately implies the first bound in \eqref{eq:bound_nu}. From \eqref{eq:bounds_imaginary_part}, we conclude the second estimate in \eqref{eq:bound_nu} and the third estimate in \eqref{eq:bound_Theta_R}. 
This completes the proof of Proposition~\ref{pro:quadratic_for_shape_analysis}.  
\end{proof}

\begin{proof}[Proof of Lemma~\ref{lem:derivation_quadratic_equation}]
Owing to the Hölder-continuity of $M$, we conclude that $M(z)$ is invertible and satisfies \eqref{MDE} for all $z \in\overline{\Dtheta}$. Hence, evaluating \eqref{MDE} at $z=\tau_0+\omega$ and $z=\tau_0$, 
computing their difference and introducing $M\defeq M(\tau_0)$ as well as $\Delta \defeq M(\tau_0 + \omega) - M$, we obtain 
\begin{equation} \label{eq:difference_mde} 
 \BO[\Delta] = \cA[\Delta, \Delta] + \omega M^2 + \omega \cK[\Delta], \qquad \qquad \cK[\Delta] \defeq \frac{1}{2} (M \Delta + \Delta M ). 
\end{equation}
In order to compute $R = \cQ[\Delta]$, we apply $\BO^{-1}\cQ$ to \eqref{eq:difference_mde}, use 
$\Delta = \Theta B + R$ and, owing to the Hölder-continuity of $M$, $\norm{\Delta(\omega)}
\lesssim \abs{\omega}^{1/2}$ and $\abs{\Theta(\omega)}\lesssim \abs{\omega}^{1/2}$, 
find $\delta_* \sim 1$ such that
\begin{equation} \label{eq:bounds_r}
 \norm{R(\omega)} \lesssim \abs{\omega}, \qquad \qquad \norm{\Im R(\omega)} \lesssim \abs{\omega}^{1/2} \Im \Theta(\omega)  
\end{equation}
for all $\omega \in [-\delta_*, \delta_*]$. Here, in order to estimate $\Im R$, we used that $M=M^*$ and, hence, $\cA[B,B] = \cA[B,B]^*$ as $\tau_0\in \pt\supp\dens$. 
This shows the second estimate in \eqref{eq:bounds_imaginary_part}. 

We apply $\scalar{P}{\genarg}$ to \eqref{eq:difference_mde} and use the decomposition $\Delta = \Theta B + R$ as well 
as $\BO[B] = \beta B$ which yield 
\[ \Theta \beta \scalar{P}{B}  = \omega \scalar{P}{M^2} + \Theta^2 \scalar{P}{\cA[B,B]} + e, \qquad 
 e  \defeq \scalar{P}{ \Theta ( \cA[B,R] + \cA[R,B]) + \cA[R,R]} + \omega \scalar{P}{\cK[\Delta]}.  \] 
From \eqref{eq:bounds_r}, we conclude 
\[ \abs{e(\omega)} \lesssim \abs{\omega}^{3/2}, \qquad \abs{\Im e(\omega)} \lesssim \abs{\omega} \Im \Theta(\omega).  \] 
This establishes the quadratic equation as well as the missing bounds on $e$ 
and, thus, completes the proof of Lemma~\ref{lem:derivation_quadratic_equation}. 
\end{proof}

\subsection{Shape analysis} 
 \label{subsec:proof_dens_close_to_regular_edge_below} 

In this section, we conclude Theorem~\ref{thm:dens_close_to_regular_edge} from Proposition~\ref{pro:quadratic_for_shape_analysis}. 

\begin{proof}[Proof of Theorem~\ref{thm:dens_close_to_regular_edge}]
We recall that $\sigma = \mu_2 = \scalar{P}{M\SS[B] B + B \SS[B]M}/2$ as in the proof of Proposition~\ref{pro:quadratic_for_shape_analysis}  and $\abs{\sigma} \sim 1$ by Proposition~\ref{pro:sigma_sim_1} (i).
We will show that there is $\delta_* \sim 1$ such that 
\begin{equation} \label{eq:dens_around_tau_0_precise} 
\dens(\tau_0 + \omega) = \begin{cases} 
\displaystyle \frac{\pi^{1/2}}{\abs{\sigma}^{1/2}} \abs{\omega}^{1/2} + \ord(\abs{\omega}), \quad & \text{if } \sign \omega = \sign \sigma, \\  
0, \qquad  & \text{if }\sign \omega = -\sign \sigma, 
\end{cases} 
\end{equation} 
for all $\omega \in [-\delta_*, \delta_*]$. This directly implies Theorem~\ref{thm:dens_close_to_regular_edge} with $c = \sqrt{\pi/\abs{\sigma}}$ as we conclude $\sigma <0$ from \eqref{sqrt growth} 
and \eqref{eq:dens_around_tau_0_precise}.  

We now compute $\Theta(\omega)$ in \eqref{eq:decomposition_diff_M} by identifying the correct solution of \eqref{eq:quadratic_equation}. 
The general quadratic equation $\Omega(\zeta)^2 + \zeta  = 0$ with $\zeta \in \C$ has two solutions:
\[ \Omega_\pm(\zeta) = \pm \begin{cases} \ii \zeta^{1/2}, & \text{if } \Re \zeta \geq 0, \\ 
 -(-\zeta)^{1/2}, & \text{if } \Re \zeta < 0, \end{cases} 
\] 
where $\zeta^{1/2}$ denotes the standard branch of the square root with the branch cut $(-\infty,0)$.  

Since $\Theta(\omega)$ is a continuous function of $\omega$ and $\abs{\nu(\omega)} < 1$ for all 
$\omega \in [-\delta_*, \delta_*]$ for $\delta_* \sim 1$ sufficiently small due to the first bound in \eqref{eq:bound_nu}, 
we conclude from \eqref{eq:quadratic_equation} that there are $p, q \in \{ +,-\}$ such that 
\begin{equation} \label{eq:Theta_omega_computed}
 \Theta(\omega) = \Omega_{p} (\Lambda(\omega))\charfunc(\omega/\sigma <0) + \Omega_q (\Lambda(\omega)) \charfunc (\omega/\sigma \geq 0), \qquad \qquad 
\Lambda(\omega) \defeq \frac{\pi \omega}{\sigma}(1  + \nu(\omega)) 
\end{equation}
for all $\omega \in [-\delta_*, \delta_*]$. 

We now show that $q=+$ by a proof by contradiction. We assume $q=-$. For $\omega/\sigma \geq 0$, we have 
\[ \Im \Omega_-(\Lambda(\omega)) = - \bigg(\frac{\pi \omega}{\sigma}\bigg)^{1/2}  + \ord\Big( \abs{\nu(\omega)}\abs{\omega}^{1/2}\Big). \]
For sufficiently small $\omega$ we thus obtain $\Im \Omega_-(\Lambda(\omega)) < 0$ in contradicition to $\Im \Theta(\omega) \geq 0$ from \eqref{eq:bound_Theta_R}. This implies $q = +$.

Next, we prove that $\Im \Theta(\omega) = 0$ for all $\omega \in I_{\delta_*}$ with $\delta_* \sim 1$ sufficiently small, where $I_{\delta_*} \defeq \{ \omega \in \R\colon \sign \omega = - \sign \sigma$, ~ $\abs{\omega} \leq \delta_*\}$. 
We will not determine $p$ in \eqref{eq:Theta_omega_computed} but rather show that $\Im \Theta = 0$ on $I_{\delta_*}$ 
for either choice of $p$ (In fact, $p = +$ can be shown \cite[Proposition~7.10~(ii)]{shapepaper}). 
By possibly shrinking $\delta_* \sim 1$, we get 
\[ \abs{\Re \Omega_\pm(\Lambda(\omega))} \sim \abs{\omega}^{1/2}  \] 
as $\sigma \in \R$ and $\abs{\sigma} \sim 1$. 
Therefore, taking the imaginary part of \eqref{eq:quadratic_equation} and using the second bound in \eqref{eq:bound_nu}, \eqref{eq:Theta_omega_computed} and $\sigma \in \R$ yield 
\[ \abs{\omega}^{1/2} \Im \Theta(\omega) \lesssim \abs{\omega} \Im \Theta(\omega)  \] 
for all $\omega \in I_{\delta_*}$. 
If $\delta_* \sim 1$ is sufficiently small then we obtain $\Im \Theta(\omega) = 0$
for all $\omega \in I_{\delta_*}$. 

We now take the imaginary part of \eqref{eq:decomposition_diff_M} and apply $\avg{\genarg}$. Hence, we obtain 
\begin{equation} \label{eq:expansion_dens_around_tau_0}
 \dens(\tau_0 + \omega) = \Im \Theta(\omega) \pi^{-1} \avg{B} + \pi^{-1} \avg{\Im R(\omega)}
 = \Im \Theta(\omega) + \ord\Big( \abs{\omega}^{1/2} \Im \Theta(\omega)\Big) 
\end{equation}
for all $\omega \in [-\delta_*, \delta_*]$. Here, we used $B=B^*$ in the first step and 
$\avg{B} = \pi$ by \eqref{eq:expansion_b}, \eqref{eq:b_0_l_0_approx}, \eqref{eq:f_u_qq_star} and \eqref{eq:lim_eta_dens_inverse} as well 
as the third bound in \eqref{eq:bound_Theta_R} in the second step. 

Since $q=+$ in \eqref{eq:Theta_omega_computed}, we can bound $\Im \Theta(\omega) = \Im \Omega_+(\Lambda(\omega))$ 
directly in \eqref{eq:expansion_dens_around_tau_0} to obtain the first case in \eqref{eq:dens_around_tau_0_precise}.
Since $\Im \Theta(\omega)=0$ for all $\omega \in I_{\delta_*}$, 
\eqref{eq:expansion_dens_around_tau_0} implies the second case in \eqref{eq:dens_around_tau_0_precise}. 
This completes the proof of \eqref{eq:dens_around_tau_0_precise} and, thus, the one of Theorem~\ref{thm:dens_close_to_regular_edge}.  
\end{proof} 

\subsection{Proof of Proposition~\ref{prop stability reference}} \label{subsec:proof_of_pro_stab_mde} 

We have now established all results which are necessary for the proof of Proposition~\ref{prop stability reference}. 

\begin{proof}[Proof of Proposition~\ref{prop stability reference}]
Claims \eqref{prop31 unique} and \eqref{prop31 dens} follow directly from \cite{MR2376207} and \cite{2016arXiv160408188A}. 

Part \eqref{prop sqrt edge} is a direct consequence of Theorem~\ref{thm:dens_close_to_regular_edge} and the Stieltjes transform representation of $\avg{M(z)}$, i.e.~
\begin{equation} \label{eq:Stieltjes_representation}
 \avg{M(z)} = \int_\R \frac{\dens(\tau)}{\tau - z}\, \di \tau 
\end{equation}
for $z \in \Hb$  (this simple calculation can be found, e.g.~in Corollary~A.1 in \cite{2015arXiv150605095A}). 
  
For the proof of \eqref{1-CMS bound with flatness}, we first remark that 
\eqref{prop sqrt edge} implies $\dens(z) + \eta \dens(z)^{-1} \sim \sqrt{\abs{\tau- \tau_0} + \eta}$ for all $z \in\Hb$ satisfying $\abs{z-\tau_0} \leq \delta_*$. 
Thus, Theorem~\ref{thm:B_inverse_sharp_edge} yields the first bound in \eqref{1-CMS bound with flatness}. 
Owing to \eqref{eq:B_inverse_Q_norm}, we have $\normsp{\BO^{-1}\cQ} \lesssim 1$. 
Moreover, we choose $P$ and $B$ as in Corollary~\ref{coro:eigenvector_expansion}. This completes the proof of the second bound in \eqref{1-CMS bound with flatness} due to \eqref{eq:b_l_bounded}. 

Moreover, $\abs{\sigma} \sim 1$ by Proposition~\ref{pro:sigma_sim_1}. 
Hence, we conclude $\abs{\scalar{P}{M\SS[B]B}} \sim 1$ from Lemma~\ref{lem:coefficient_quadratic_term}. 
Furthermore, owing to \eqref{eq:b_l_bounded}, we have $\abs{\scalar{P}{B}} \sim 1$. 
Thus, since $\sigma \in \R$ and $\abs{\sigma} \sim 1$ we get from \eqref{eq:expansion_beta_scalar_l_b} that 
\[ \abs{\beta} \sim \abs{\beta\scalar{P}{B}} \sim \varrho + \eta \varrho^{-1} \sim \sqrt{\abs{\tau - \tau_0} + \eta} . \] 
This completes the proof of Proposition~\ref{prop stability reference}. 
\end{proof}

\section{Band rigidity}  \label{sec:band_rigidity} 

Within this section we establish band rigidity for correlated random matrices $H$.  This topological rigidity phenomenon asserts that the number of eigenvalues of $H$ within a spectral band, i.e.~a connected component of $\supp \varrho$, does not fluctuate and is accurately predicted by the self-consistent density of states with high probability. On the level of the MDE this phenomenon is reflected by the \emph{band mass formula} \eqref{band mass formula} below, guaranteeing that $N\varrho$ assigns only integer values to each band. In particular, small continuous deformations of the data $(A,\SS)$ of the MDE cannot change these values. 
\begin{proposition}[Band mass formula] 
\label{prp:Band mass formula}
For $\tau \in \R\setminus \supp \varrho$ the integrated self-consistent density of states satisfies
\begin{equation}\label{band mass formula}
\int_{-\infty}^\tau \varrho(x)\dd x = \frac{1}{N} \abs{\mathrm{Spec}(M(\tau))\cap (-\infty,0)}.
\end{equation}
In particular, $N \int_{-\infty}^\tau \varrho(x)\dd x$ is an integer.
\end{proposition}

Before we prove Proposition~\ref{prp:Band mass formula} we show how it is used to establish band rigidity for $H$.
\begin{proof}[Proof of Corollary \ref{cor rigidity}]
We begin with the proof of \eqref{det num of evs} and consider a flow that interpolates between $H=H_0$ and a deterministic matrix $H_1$. We fix $\tau\not\in\supp\varrho$ with $\epsilon\defeq \dist(\tau,\supp\varrho)>0$ and set
\begin{equation}\label{Ht flow}
H_t \defeq \sqrt{1-t}\mspace{2mu} W+A_t, \quad A_t \defeq A-t\mspace{2mu}\SS[M(\tau)], \quad \SS_t\defeq (1-t)\SS, \quad t\in[0,1].
\end{equation}
The MDE corresponding to $H_t$ is 
\begin{equation}\label{MDEt}
\id+(z-A_t+\SS_t[M_t(z)])M_t(z)=0
\end{equation}
with data $(A_t,\SS_t)$, solution $M_t(z)$ and self-consistent density of states $\varrho_t$. We refer to this $t$-dependent MDE as $\mathrm{MDE}_t$. It is designed in such a way that $M_t(\tau)$ at the fixed spectral parameter $z=\tau$ is kept constant at $t$ varies. Moreover, by the following lemma, whose proof we postpone, $\tau$ stays away from the self-consistent spectrum along the flow.
\begin{lemma} 
\label{lmm:E away from spectrum} Let $\epsilon\defeq \dist(\tau,\supp\varrho)>0$ and $M_t$ be the solution to $\mathrm{MDE}_t$ \eqref{MDEt}.
Then $\dist(\tau,\supp \varrho_t) \ge_\epsilon 1$ and $\lim_{\eta \downarrow 0}M_t(\tau+\mathrm{i}\eta)=M(\tau)$ for all $t \in [0,1]$.
\end{lemma}

We will now show that along the flow, with overwhelming probability, no eigenvalue crosses the spectral parameter $\tau$. More precisely we claim that
\begin{equation}\label{No evs cross E}
\P\Big(\tau\in\Spec H_t \text{ for some $t\in[0,1]$}\Big)\le_\epsilon N^{-D}
\end{equation}
for any $D>0$. Since $H_0=H$ and $H_1=A-\SS[M(\tau)]$, \eqref{No evs cross E} implies that with overwhelming probability
\begin{equation*}
\abs{\Spec H \cap (-\infty,\tau)}=\abs{\Spec(A-\SS[M(\tau)]-\tau) \cap (-\infty,0)}=N \braket{\bm 1_{(-\infty,0)}(M(\tau))},
\end{equation*}
where the last identity used the the MDE \eqref{MDE} at $z=\tau$.
 Now \eqref{det num of evs} follows from the band mass formula \eqref{band mass formula}, i.e.~from $\braket{\bm 1_{(-\infty,0)}(M(\tau))} = \int_{-\infty}^\tau \varrho(\lambda) \diff\lambda$. 
 
It remains to show \eqref{No evs cross E}. We first consider the regime of values $t$ close to $1$. Since $\tau$ is separated away from $\supp\varrho$, and $M(\tau)$ is bounded we conclude from \eqref{MDE} at $z=\tau$ that the spectrum of $A-\SS[M(\tau)]$ is also separated away from $\tau$. 
Moreover, applying \cite[Corollary 2.3]{2017arXiv170510661E} to $H=W$ yields $\norm{W} \leq C$ with overwhelming probability as the corresponding self-consistent density of states has compact support by Proposition \ref{prop stability reference}\eqref{prop31 dens}.  
Since therefore $H_t$ is a small perturbation of $A-\SS[M(\tau)]$ as long as $t$ is close to $1$, we conclude that the spectrum of $H_t$  is bounded away from $\tau$ as well for every fixed $t \ge 1-c$ for some small enough constant $c>0$. We are thus left with the regime $t \le 1-c$, where the flatness condition from Assumption \ref{assumption flatness} for $H_t$ is satisfied. In this regime we use  \cite[Corollary 2.3]{2017arXiv170510661E}  again.  Since $\dist(\tau, \supp \varrho_t) \ge_\epsilon  1$ this corollary implies that the spectrum of $H_t$ is bounded away from $\tau$ with overwhelming probability for every fixed $t \le 1-c$. Applying a discrete union bound in $t$ together with the Lipschitz continuity of the eigenvalues in $t$ for the flow \eqref{Ht flow} on the set $\norm W\le C$ yields \eqref{No evs cross E}.

Finally, \eqref{eq rigidity} follows from the optimal local law as in the proof of  Theorem \ref{theorem local law} and Corollary \ref{cor no eigenvalues outside} above. This time, however, \eqref{det num of evs} ensures that there is no mismatch between location and label of eigenvalues close to internal edges.  In the spectral bulk this potential discrepancy between label and location does not matter as \eqref{eq rigidity} allows for an $N^\epsilon$-uncertainty. At the spectral edge, however, neighbouring eigenvalues can lie on opposite sides of a spectral gap and we need \eqref{det num of evs} to make sure that each eigenvalue has, with high probability, a definite location with respect to the spectral gap. 
\end{proof}

\begin{proof}[Proof of Lemma~\ref{lmm:E away from spectrum}]
Note that $M(z)$ is analytic and bounded away from the self-consistent spectrum because it admits a Stieltjes transform representation (cf.~Proposition~2.1 of \cite{2016arXiv160408188A}). 
We consider $\mathrm{MDE}_t$ \eqref{MDEt} at a spectral parameter $\tau+\zeta$ with some $\zeta \in \mathbb{H}$ such that $\abs{\zeta}\ll1$ and subtract it from $\mathrm{MDE}_t$ at spectral parameter $\tau$. Properly symmetrised the resulting quadratic equation for $\Delta=\Delta(\zeta)=M_t(\tau+\zeta)-M(\tau)$ takes the form
\begin{equation}\label{E stability equation}
\mathcal{B}_t[\Delta] = \zeta M^2 +\frac{\zeta}{2}(M\Delta+\Delta M) +(1-t)\mathcal{A}[\Delta,\Delta],
\end{equation} 
where $M=M(\tau)$, $\mathcal{A}$ is as in \eqref{eq:def_A} and $\mathcal{B}_t=\Id - (1-t)\cC_M \SS$ is the stability operator.
We will see that equation \eqref{E stability equation} is linearly stable in the sense that $\norm{\mathcal{B}_t^{-1}}\le_\epsilon 1$ uniformly in $t$. Note that the terms containing $\Delta$ on the right hand side are lower order. Thus we may apply the implicit function theorem to show that $\Delta(\zeta)$ is an analytic function for sufficiently small $\zeta$ with $\Delta(\zeta) = \zeta \mspace{2mu}\mathcal{B}_t^{-1}[M^2] + \landauO{\abs{\zeta}^2}$. In particular, it extends to small $\zeta \in \C$.  Since $M=M(\tau)$ is self-adjoint and $\mathcal{B}_t^{-1}$ preserves the cone of positive definite matrices, $M + \Delta(\zeta)$ coincides for any small $\zeta \in \mathbb{H}$ with the unique solution to $\mathrm{MDE}_t$ with positive definite imaginary part. But since $\Delta(\zeta)$ is analytic in $\zeta$ for any small enough $\zeta$, even with negative imaginary part, 
 $M_t(z)$ can be analytically extended to a $t$-independent neighbourhood of $\tau$ in $\C$. Furthermore, since $\mathcal{B}_t$ and $R \mapsto \mathcal{A}[R,R]$ preserve the space of self-adjoint matrices, this extension takes self-adjoint values on the real line. Thus for every $t$ the density $\varrho_t = \frac{1}{\pi}\braket{ \Im M_t}$ vanishes in a neighbourhood  of $\tau$, i.e.~$\dist(\tau,\supp \varrho_t) \ge_\epsilon 1$. 

To show the bound on $\mathcal{B}_t^{-1}$
we use the symmetrisation \eqref{eq:representation_BO} with the self energy operator $\SS_t = (1-t)\SS$ to see that
\begin{equation}\label{bound on cal B t}
\norm{\mathcal{B}_t^{-1}}_{\mathrm{sp}}\le_\epsilon \norm{(\mathcal{C}_U^*-\mathcal{F}_t)^{-1}}_{\mathrm{sp}}\lesssim  \frac{1}{1-(1-t)\norm{\mathcal{F}}_{\mathrm{sp}}},
\end{equation}
where $U$ is unitary and $\mathcal{F}_t=(1-t)\mathcal{F}$ with the self-adjoint operator $\mathcal{F}$ from \eqref{eq:def_F}. 
Exactly as in the proof of Lemma~\ref{lem:prop_F_small_dens} the boundedness of $\mathcal{B}_t^{-1}$ in the $\norm{\cdot}_{\mathrm{sp}}$-norm also implies $\norm{\mathcal{B}_t^{-1}}\le_\epsilon 1$. Thus it remains to show that the right hand side of \eqref{bound on cal B t} is bounded. 
 For this purpose we apply the lower bound on $1-\norm{\mathcal{F}}_{\mathrm{sp}} \ge_\epsilon 1$ from \cite[Lemma 3.6]{1706.08343}, finishing the proof of the lemma.
\end{proof}

\begin{proof}[Proof of Proposition \ref{prp:Band mass formula}] Let $\epsilon\defeq \dist(\tau, \supp \varrho) > 0$.  Again we make use of $\mathrm{MDE}_t$ \eqref{MDEt}. Recall that $M(\tau)$ solves $\mathrm{MDE}_t$ at spectral parameter $\tau$, which stays away from the self-consistent spectrum by Lemma~\ref{lmm:E away from spectrum}. 
 
Since $M_t(z)$ is the Stieltjes transform of a matrix valued measure on $\supp \varrho_t$ it can be analytically extended to $\C\setminus \supp \varrho_t$, a set that contains the spectral parameter $\tau$ for which $M_t(\tau)=M(\tau)$ by the lemma.
When $\varrho$ and $M(\tau)$ are replaced by $\varrho_t$ and $M_t(\tau)$, respectively, in \eqref{band mass formula} then clearly this identity holds at time $t=1$
since $M_1(z)=(A-\SS[M(\tau)]-z)^{-1}=(\tau+M(\tau)^{-1}-z)^{-1}$ is the resolvent of the self-adjoint matrix $\tau+M(\tau)^{-1}$. 
As $M_t(\tau) = M(\tau)$, it  suffices to establish that the left-hand side of 
\eqref{band mass formula} with $\varrho$ replaced by $\varrho_t$ does not change along the flow.

To show that the left hand side is independent of $t$, we differentiate the contour integral representation 
\[
\int_{-\infty}^\tau \varrho_t(x)\dd x = -\oint \frac{\mathrm{d}z}{2\pi \mathrm{i}}\braket{M_t(z)},
\] 
where the contour encircles $[\min \supp \varrho_t,\tau)$ counterclockwise, passing through the real line only at $\tau$ and to the left of $\inf_t \min \supp \varrho_t$. With $M_t=M_t(z)$ we find
\[
\frac{\dd}{\dd t}\oint \braket{M_t}\dd z= \oint  \braket{(\mathcal{C}_{M^*_t}^{-1}-\SS_t)^{-1}[\id],\mathcal{S}[M(\tau)-M_t]}
=\oint \partial_z \Big(\braket{M_t\SS[M(\tau)]} -\frac{1}{2}\braket{M_t\SS[M_t]}\Big) \dd z =0,
\]
where the formula  $(\mathcal{C}_{M_t}^{-1}-\SS_t)[\partial_t M_t] = \SS[M(\tau)-M_t]$, used in the first identity, is obtained by differentiating  $\mathrm{MDE}_t$  with data \eqref{Ht flow} with respect to $t$ and the formula  $(\mathcal{C}_{M_t}^{-1}-\SS_t)[\partial_zM_t] = \id$, used in the second identity, follows from differentiating  \eqref{MDEt} with respect to $z$.
\end{proof}

\section{Proof of Universality}\label{sec proof edge univ}
In order to prove Theorem \ref{thm edge univ}, we define the Ornstein Uhlenbeck (OU) process starting from $H=H_0$ by
\begin{equation}\label{ou}
\diff H_t = - \frac{1}{2}(H_t-A)\diff t+\Sigma^{1/2}[\diff B_t], \qquad \Sigma[R]\defeq \E W\Tr (W R),
\end{equation}
where $B_t$ is a matrix of, up to symmetry, independent (real or complex, depending on the symmetry class of $H$) Brownian motions and $\Sigma^{1/2}$ is the  square root of the positive definite operator $\Sigma\colon \C^{N \times N} \to \C^{N \times N} $. We note that the same process has already been used in \cite{MR3629874,2016arXiv160408188A,2017arXiv170510661E} to prove bulk universality. The proof now has two steps: Firstly, we will prove edge universality for $H_t$ if $t\gg N^{-1/3}$ and then we will prove that for $t\ll N^{-1/6}$, the eigenvalues of $H_t$ have the same $k$-point correlation functions as those of $H=H_0$. 

\subsection{Dyson Brownian Motion}
The process \eqref{ou} can be integrated, and we have 
\begin{equation*}
H_t-A = e^{-t/2}(H_0-A) + \int_0^t e^{(s-t)/2}\Sigma^{1/2}[\diff B_s],\qquad \int_0^t e^{(s-t)/2}\Sigma^{1/2}[\diff B_s]\sim\NN(0,(1-e^{-t})\Sigma).
\end{equation*}
The process is designed in such a way that it preserves expectation $\E H_t=A$ and covariances $\Cov{h_{ab}^t,h_{cd}^t}=\Cov{h_{ab},h_{cd}}$ along the flow. Due to the fullness Assumption \ref{assumption fullness} there exists a constant $c>0$ such that $(1-e^{-t})\Sigma-c t \Sigma^{\GUE/\GOE}\ge 0$ for $t\le 1$, where $\Sigma^{\GOE/\GUE}$ denotes the covariance operator of the GOE/GUE ensembles. It follows that we can write 
\begin{equation*}
H_t = \widetilde H_t + \sqrt{ct} U, \qquad \kappa_t = \kappa - c t \kappa^{\GOE/\GUE},\qquad \E \widetilde H_t=A, \qquad U \sim \text{GOE/GUE},
\end{equation*}
where $\kappa_t$ here denotes the cumulants of $\wt H_t$, and $U$ is chosen to be independent of $\widetilde H_t$. Due to the  fact that Gaussian cumulants of degree more than $2$ vanish, it is easy to check that $H_t,\widetilde H_t$ satisfy the assumptions of Theorem \ref{theorem local law} uniformly in, say, $t\le N^{-1/10}$. 
From now on we fix $t=N^{-1/3+\epsilon}$ with some small $\epsilon>0$. 

Since the MDE is purely determined by the first two moments of the corresponding random matrix, it follows that $G_t\defeq( H_t-z)^{-1}$ is close to the same $M$ in the sense of a local law for all $t$. For $\widetilde G_t\defeq (\widetilde H_t-z)^{-1}$ we have the MDE
\begin{equation}\label{MDE t}
\id+(z-A+\SS_t[M_t])M_t =0, \qquad \SS_t\defeq\SS- ct \SS^{\GOE/\GUE}
\end{equation}
that can be viewed as a perturbation of the original MDE with $t=0$.
 The corresponding self-consistent density of states is $\varrho_t(\tau)\defeq \lim_{\eta\searrow 0} \Im \braket{M_t(\tau+\ii\eta)}/\pi$. 
The fact that  $M_t$ remains bounded uniformly in $t\le N^{-1/10}$ follows from a similar (but much simpler) argument as those leading to the local law in Section~\ref{sec proof local law}. 
The analogue of \eqref{G-M D eq} with $G$ replaced by $M_t(z)$ is obtained by subtracting \eqref{MDE} from \eqref{MDE t} and the analogue of the error term $D$ is trivially controlled by $t$.  
The details are presented in the MDE perturbation result in \cite[Proposition 10.1]{shapepaper} with $S=\SS$, $S_t = \SS_t$ and $a_t = A$ as 
the condition on $S_t$ in \cite[Eq.~(10.1)]{shapepaper} is obviously satisfied for this choice of $S_t$
due to $\norm{\SS^{\GOE/\GUE}[R]} \lesssim \braket{R}$ for all positive semidefinite matrices $R$.
   In particular the shape analysis from Section \ref{sec:dens_close_to_edge} also applies to $M_t$. 

The Stieltjes transforms of the free convolutions of the empirical spectral density of $\widetilde H_t$ and $\varrho_t$ with the semicircular distribution generated by $\sqrt{ct} U$ are given implicitly as the unique solutions to the equations
\begin{equation*}
\widetilde m_{\fc}^t(z) = \braket{\widetilde G_t(z+ct \widetilde m_{\fc}^t(z))}, \qquad m_{\fc}^t(z)=\braket{M_t(z+ct m_{\fc}^t(z))}.
\end{equation*}
We denote the corresponding right-edges close to $\tau_0$ by $\widetilde \tau_t$ and $\tau_t$. By differentiating the defining equations for $m_\fc^t$ and $\wt m_\fc^t$ we find
\begin{subequations}
 \begin{equation}\label{mfc diff}
 \frac{(m_\fc^t)'(z)}{1+ct (m_\fc^t)'(z)} = \braket{M_t'(\xi_t(z))}, \quad \frac{(\wt m_\fc^t)'(z)}{1+ct (\wt m_\fc^t)'(z)} = \braket{\wt G_t'(\wt\xi_t(z))}, \quad \frac{(m_\fc^t)''(z)}{(1+ct (m_\fc^t)'(z))^3} = \braket{M_t''(\xi_t(z))},
 \end{equation}
where $\xi_t(z)\defeq z+ct m_\fc^t(z)$ and $\widetilde\xi_t(z)\defeq z+ct\widetilde m_\fc^t(z)$. From the first two equalities in \eqref{mfc diff} we conclude
\begin{equation}
1=ct \braket{M_t'(\xi_t(\tau_t))},\qquad 1 = ct \braket{\wt G_t'(\wt\xi_t(\wt \tau_t))},\label{edge impl eq}
\end{equation}
\end{subequations}
by considering the $z\to \tau_t$  and $z\to \wt \tau_t$ limits and that $(m^t_\fc)'$, $(\wt m^t_\fc)'$ blow up at the edge due to the well known square root behaviour of the density along the semicircular flow. We now compare the edge location and edge slope of the densities $\varrho_\fc^t$ and $\wt\varrho_\fc^t$ corresponding to $m_\fc^t$ and $\wt m_\fc^t$ with that of $M$. Very similar estimates for deformed Wigner ensembles have been used in \cite{2017arXiv171203936H}. We split the analysis into four claims.

\subsubsection*{Claim 1} $\abs{\tau_t-\tau_0}\lesssim t/N$.
Using that $\SS^\GUE[R]=\braket{R}$, $\SS^\GOE[R]=\braket{R}+R^t/N$ and \eqref{MDE t} evaluated at $\xi_t(z)$, we find using the boundedness of $M_t$,
\begin{equation*}
\id+(z-A+\SS[M_t(\xi_t(z))])M_t(\xi_t(z)) = ct\Big(\SS^{\GOE/\GUE}[M_t(\xi_t(z))]-\braket{M_t(\xi_t(z))}\Big)M_t(\xi_t(z)) =\landauO{\frac{t}{N}}. 
\end{equation*}
It thus follows that $M_t(\xi_t(z))$ approximately satisfies the MDE for $M$ at $z$. By using the first bound in Proposition \ref{prop stability reference}\eqref{1-CMS bound with flatness} expressing the stability
of the MDE against small additive perturbations it follows that
\begin{equation}\label{mfc M diff}
\abs{m_\fc^t(z)-\braket{M(z)}}=\abs{\braket{M_t(\xi_t(z))-M(z)}} \lesssim \frac{t}{N\sqrt{\eta+\dist(\Re z,\partial\supp\varrho)}}\le \frac{t}{N\sqrt{\dist(\Re z,\partial \supp\varrho)}}.
\end{equation}
Suppose first that $\tau_0=\tau_t+\delta$ for some positive $\delta>0$. Then $\sqrt{\delta}\lesssim\Im\braket{M(\tau_t+\delta/2)}\lesssim t/N\sqrt{\delta}$, where the first bound follows from the square root behaviour of $\varrho$ at the edge $\tau_0$, while the second bound comes from \eqref{mfc M diff} at $z=\tau_t+\delta/2$ and $\Im m_\fc^t(\tau_t+\delta/2)=0$. We thus conclude $\delta\lesssim t/N$. If on the contrary $\tau_0=\tau_t-\delta$ for some $\delta>0$, then with a similar argument $\sqrt{\delta}\lesssim\Im m_\fc^t(\tau_0+\delta/2)\lesssim t/N$ and we have $\delta\lesssim t/N$ also in this case and the claim follows.

\subsubsection*{Claim 2} $\abs{\gamma_t- \gamma}\lesssim (t/N)^{1/4}$, where $\gamma=\gamma_\mathrm{edge}$ from Definition \ref{def regular edge}. From the third equality in \eqref{mfc diff} we can relate the edge-slope of $m_\fc^t$ to $M_t''$. Indeed, if $\gamma_t^{3/2}$ denotes the slope, i.e.~$\varrho_\fc^t(x)=\gamma_t^{3/2}\sqrt{(\tau_t-x)_+}/\pi+\landauo{\tau_t-x}$, then using the elementary integrals
\begin{equation*}
\lim_{\eta\to0}\eta^{1/2}\int_0^\infty \frac{\sqrt{x}/\pi}{(x-\ii\eta)^2}\diff x = \frac{\ii^{1/2}}{2}, \qquad \lim_{\eta\to0}\eta^{3/2}\int_0^\infty \frac{\sqrt{x}/\pi}{(x-\ii\eta)^3}\diff x = \frac{\ii^{3/2}}{8}
\end{equation*}
we obtain the precise divergence asymptotics of the derivatives $(m_\fc^t)'(z)$ and $(m_\fc^t)''(z)$ as $z=\tau_t+\ii\eta\to \tau_t$ and conclude 
\begin{equation*}
\frac{2}{\gamma_t^3}=\lim_{z\to \tau_t} \frac{(ct)^3 (m_\fc^t)''(z)}{(1+ct(m_\fc^t)'(z))^3} = (ct)^3 \braket{M_t''(\xi_t(\tau_t))}, \quad\text{i.e,}\quad \gamma_t = \frac{\big(\braket{M_t''(\xi_t(\tau_t))}/2\big)^{-1/3}}{ct}.
\end{equation*}
We now use \eqref{mfc M diff} at, say, $z=x\defeq \tau_0-\sqrt{t/N}$. By Claim 1 we have $\tau_t-x\sim \sqrt{t/N}$ and thus 
\begin{equation*}
\gamma_t^{3/2} = \frac{\Im m_\fc^t(x)}{\sqrt{\tau_t-x}} + \landauO{(t/N)^{1/4}} = \frac{\Im\braket{M(x)}}{\sqrt{\tau_t-x}} + \landauO{(t/N)^{1/4}} = \frac{\Im\braket{M(x)}}{\sqrt{\tau_0-x}} + \landauO{(t/N)^{1/4}}=\gamma^{3/2} +\landauO{(t/N)^{1/4}},
\end{equation*}
where we used Claim 1 again in the third equality. This completes the proof of the claim.

\subsubsection*{Claim 3} $\abs[0]{\wt \tau_t-\tau_t}\prec 1/Nt$. Since $M_t$ has a square root edge at some $\widehat \tau_t$, it follows from the first equality in \eqref{edge impl eq} that $\xi_t(\tau_t)-\widehat \tau_t \sim t^2$. Using rigidity in the form of Corollary \ref{cor rigidity} for the matrix $\wt H_t$ to estimate $\wt G'_t$ from below at a spectral parameter outside of the support, we have the bound
\begin{equation*}
ct=\abs[1]{\braket{\wt G_t'(\wt\xi_t(\wt \tau_t))}}^{-1}\prec \abs[1]{\wt\xi_t(\wt \tau_t)-\widehat \tau_t}^{1/2}. 
\end{equation*}
Consequently using the local law in the form of Lemma \ref{hs local law} it follows that
\begin{equation*}
\abs[1]{\braket{M_t'(\wt\xi_t(\wt \tau_t))}} = 1/ct + \mathcal O_\prec(1/Nt^4)\sim 1/t,
\end{equation*}
whence $\wt \xi_t(\wt \tau_t) -\widehat \tau_t\sim t^2$ where we again used the square root singularity of $\braket{M_t}$ at $\widehat \tau_t$. We can conclude, starting from \eqref{edge impl eq}, that
\begin{equation*}
\begin{split}
0 &=\braket{M_t'(\xi_t(\tau_t))}-\braket{\wt G_t'(\wt \xi_t(\wt \tau_t))} =\braket{M_t'(\xi_t(\tau_t))}-\braket{M_t'(\wt\xi_t(\wt \tau_t))}+\braket{(M_t'-\wt G_t')(\wt \xi_t(\wt \tau_t))}\\
&\sim \abs[0]{\xi_t( \tau_t)-\wt \xi_t(\wt \tau_t)}/t^3 + \mathcal O_\prec(1/Nt^4),
\end{split}
\end{equation*}
where we used that $\abs[0]{\braket{M_t''(\widehat \tau_t+rt^2)}}\sim t^{-3}$ for $c<r<C$ and the improved local law $\braket{G'-M'}\prec 1/N\kappa^2$ at a distance $\kappa\sim t^2$ away from the spectrum, as stated in Lemma \ref{hs local law}. We thus find that $\abs[0]{\xi_t(\tau_t)-\wt \xi_t(\wt \tau_t)}\prec 1/Nt$. It remains to relate this to an estimate on $\abs[0]{\tau_t-\wt \tau_t}$. We have 
\begin{equation*}
 \abs[0]{\tau_t-\wt \tau_t} \lesssim \abs[0]{\xi_t(\tau_t)-\wt\xi_t(\wt \tau_t)} + t \abs[0]{m_\fc^t(\tau_t)-m_\fc^t(\wt \tau_t)} + t\abs[0]{(m_\fc^t-\wt m_\fc^t)(\wt \tau_t)},
\end{equation*}
where we bounded the second term by $t\abs[0]{\braket{M_t(\xi_t(\tau_t))-M_t(\wt \xi_t(\wt \tau_t))}}\prec 1/Nt$ using $\abs[0]{\braket{M_t'(\widehat \tau_t+r t^2)}}\sim 1/t$ and the third term by $t\abs[0]{\braket{(M_t-\wt G_t)(\wt \xi_t(\wt \tau_t))}}\prec 1/Nt$ using the local law $t^2$ away from $\supp\varrho_t$. Thus we can conclude that $\abs[0]{\tau_t-\wt \tau_t}\prec 1/Nt$. 

\subsubsection*{Claim 4} $\abs{\gamma_t-\wt\gamma_t}\prec 1/Nt^3$. 
 We first note that $\gamma_t\sim 1$ follows from $\abs[0]{\braket{M_t''(\xi_t(\tau_t))}}\sim t^{-3}$. Therefore it suffices to estimate 
 \begin{equation*}t^{3}\abs[0]{\braket{M_t''(\xi_t(\tau_t))-\wt G_t''(\wt \xi_t(\wt \tau_t))}} \le t^{3}\abs[0]{\braket{M_t''(\xi_t(\tau_t))-M_t''(\wt \xi_t(\wt \tau_t))}} + t^{3}\abs[0]{\braket{M_t''(\wt\xi_t(\wt \tau_t))-\wt G_t''(\wt \xi_t(\wt \tau_t))}} \prec \frac{1}{Nt^3},\end{equation*}
 as follows from $\braket{M_t'''(\widehat \tau_t+rt^2)}\sim t^{-5}$ for $c<r<C$ and the local law from Lemma \ref{hs local law} at a distance of $\kappa\sim t^2$ away from the spectrum. Thus we have $\abs[0]{\gamma_t-\wt \gamma_t}\prec 1/Nt^3$. 

We now check that $\widetilde H_t$ is $\eta_\ast$-regular in the sense of \cite[Definition 2.1]{2017arXiv171203881L} for $\eta_\ast\defeq N^{-2/3+\epsilon}$. It follows from the local law that 
$c\varrho_t(z) \prec \Im\braket{\wt G_t(z)}\prec C\varrho_t(z)$ for some constants $c,C$, whenever $\Im z\ge \eta_\ast$. Now (2.4)--(2.5) in \cite{2017arXiv171203881L} follow in high probability from the assumption that $\varrho_t$ has a regular edge at $\tau_t$ . Furthermore, the absence of eigenvalues in the interval $[\tau_t+\eta_\ast, \tau_t + c/2 ]$ with high probability follows directly from Corollary \ref{cor no eigenvalues outside}. Finally, $\norm[0]{\widetilde H_t}\le N$ with high probability follows directly from $\norm[0]{\widetilde H_t}\le(\Tr\abs[0]{\widetilde H_t}^2)^{1/2}$. We can thus conclude that with high probability, $\widetilde H_t$ is $\eta_\ast=N^{-2/3+\epsilon}$ regular for any positive $\epsilon>0$. 

We denote the eigenvalues of $H_t=\widetilde H_t+c\sqrt{t} U$ by $\lambda_1^t\le \dots\le\lambda_N^t$. Then it follows from \cite[Theorem 2.2]{2017arXiv171203881L} that for $N^{-\epsilon}\ge t\ge N^{-2/3+\epsilon}$ with high probability for test functions $F\colon\R^{k+1}\to\R$ with $\norm{F}_\infty+\norm{\nabla F}_\infty\lesssim 1$ there exists some $c>0$ such that
\begin{equation}\label{landon yau}
\abs{\E\left[F\Big(\widetilde\gamma_t N^{2/3}(\lambda_{i_0}^t-\widetilde \tau_t),\dots,\widetilde\gamma_t N^{2/3}(\lambda_{i_0-k}^t-\widetilde \tau_t)\Big)\big|\widetilde H_t\right]- \E F\Big( N^{2/3}(\mu_N-2),\dots, N^{2/3}(\mu_{N-k}-2)\Big) }\le N^{-c}.
\end{equation}
By combining \eqref{landon yau} with $\abs[0]{\tau_0-\wt \tau_t}\prec N^{-2/3-\epsilon}$, $\abs[0]{\gamma-\wt\gamma_t}\prec N^{-\epsilon}$ from Claims 1--4, we obtain
\begin{equation}\label{DBM result}
\abs{\E\left[F\Big(\gamma N^{2/3}(\lambda_{i_0}^t- \tau_0),\dots,\gamma N^{2/3}(\lambda_{i_0-k}^t-\tau_0)\Big)\right]- \E F\Big( N^{2/3}(\mu_N-2),\dots, N^{2/3}(\mu_{N-k}-2)\Big) }\lesssim N^{-c} +N^{-\epsilon}
\end{equation}
for our choice of $t=N^{-1/3+\epsilon}$.

\subsection{Green's Function Comparison}
It remains to prove that the local correlation functions of $H_t$ agree with those of $H$. We want to prove that for any fixed $x_i\in\R$,
\begin{equation*}
\lim_{N\to\infty}\P\left( N^{2/3}(\lambda_{i_0-i}^t - \tau_0) \ge x_i,\, i=0,\dots,k \right)
\end{equation*}
is independent of $t$ as long as, say, $t\le N^{-1/3+\epsilon}$. We first note that the local law holds uniformly in $t$ also for $H_t$. This follows easily from the fact that the assumptions stay uniformly satisfied along the flow  because expectation and covariance are preserved while higher order cumulants also remain unchanged up to a multiplication with a $t$-dependent constant. For $l=N^{-2/3-\epsilon/3}$, $\eta=N^{-2/3-\epsilon}$, 
and smooth monotonous cut-off functions $K_i$ with $K_i(x)=0$ for $x\le i-1$ and $K_i(x)=1$ for $x\ge i$ we have
\begin{equation}\label{edge k pt E comp}
\begin{split}
&\E \prod_{i=0}^k K_{i_0-i}\left( \frac{\Im}{\pi} \int_{x_i N^{-2/3}+l}^{N^{-2/3+\epsilon}} \Tr G_t(x+\tau_0+\ii\eta)\diff x\right) - \landauO{N^{-\epsilon/9}} \le \P\left(N^{2/3}(\lambda_{i_0-i}^t-\tau_0)\ge x_i, \, i=0,\dots,k\right) \\
&\qquad\qquad\le \E \prod_{i=0}^k K_{i_0-i}\left( \frac{\Im}{\pi} \int_{x_i N^{-2/3}-l}^{N^{-2/3+\epsilon}} \Tr G_t(x+\tau_0+\ii\eta)\diff x\right)+ \landauO{N^{-\epsilon/9}}.\end{split}
\end{equation}
We note that the strategy of expressing $k$-point correlation functions of edge-eigenvalues through a regularized expression involving the resolvent was already used in \cite{MR2871147,MR3034787,MR3405746,2017arXiv171203936H} for proving edge universality. The precise formula \eqref{edge k pt E comp} has been already used, for example, in \cite[Eq.~(4.8)]{2017arXiv171203936H}. 

In order to compare the expectations in \eqref{edge k pt E comp} at times $t=0$ and $t=N^{-1/3+\epsilon}$, we claim that we have the bound
\begin{equation}\label{Xt bound}
X_y \defeq \Im\int_{y N^{-2/3}\pm l}^{N^{-2/3+\epsilon}} \Tr G_t(\tau_0+x+\ii\eta) \diff x, \quad \abs{\E g(X_{x_0},\dots,X_{x_k})\frac{\diff X_{x_j}}{\diff t}}\lesssim N^{1/6+3\epsilon}
\end{equation}
for any $0\le j\le k$ and smooth function $g$. Assuming \eqref{Xt bound}, it follows for the smooth functions $K_{j}$ and by Taylor expansion that that for $t\lesssim N^{-1/3+\epsilon}$,
\begin{equation*}
\Bigg\lvert\E \prod_{i=0}^k K_{i_0-i}\left( \frac{\Im}{\pi} \int_{x_iN^{-2/3}\pm l}^{N^{-2/3+\epsilon}} \Tr G_t(x+\tau_0+\ii\eta)\diff x\right) - \E \prod_{i=0}^k K_{i_0-i}\left( \frac{\Im}{\pi} \int_{x_iN^{-2/3}\pm l}^{N^{-2/3+\epsilon}} \Tr G_0(x+\tau_0+\ii\eta)\diff x\right)\Bigg\rvert\lesssim \frac{1}{N^{1/6-4\epsilon}}.
\end{equation*}
Together with \eqref{edge k pt E comp} we obtain for any $k,x_i$
\begin{equation}\label{gfc result}
\P\left( N^{2/3}(\lambda_{i_0-i}^t - \tau_0) \ge x_i,\, i=0,\dots,k \right) =\P\left( N^{2/3}(\lambda_{i_0-i}^0 - \tau_0) \ge x_i,\, i=0,\dots,k \right)  +\landauO{N^{-\epsilon/9}}.
\end{equation}
Eq.~\eqref{Xt bound} for $g\equiv 1$ follows from It\^o's lemma in the form \begin{equation*}
\E \frac{\diff f(H)}{\diff t} = \E \Bigg[- \frac{1}{2}\sum_\alpha w_\alpha (\partial_\alpha f)(H) +\frac{1}{2}\sum_{\alpha,\beta}\kappa(\alpha,\beta) (\partial_{\alpha}\partial_\beta f)(H) \Bigg]
\end{equation*} and the general neighbourhood cumulant expansion involving \emph{pre-cumulants}, as introduced in \cite[Proposition 3.5]{2017arXiv170510661E}. This expansion formula was a key input to the Green's function comparison argument in the spectral bulk in \cite[Corollary 2.6]{2017arXiv170510661E} for correlated matrix models under Assumptions \hyperlink{assumpCD}{(CD)}. Given the local law, Theorem \ref{theorem local law}, the extension of this proof to the edge is a routine power counting argument even for $g\not\equiv 1$ and is left to the reader. 

\begin{proof}[Proof of Theorem \ref{thm edge univ}]
The theorem follows directly from \eqref{DBM result} and \eqref{gfc result}.
\end{proof}

\appendix
\section{Auxiliary results}
\begin{proof}[Proof of Lemma \ref{lemma ast norm bounds}]
From (70a)--(70b) in \cite{2017arXiv170510661E} we have\footnote{C.f.~Remark \ref{ast norm apology} for the applicability of these bounds in the present setup.} 
\begin{subequations}
\begin{equation}
\norm{M\SS[R]R}_\ast \lesssim N^{1/2K} \norm{R}_\ast^2, \qquad \norm{MR}_\ast \lesssim N^{1/2K}\norm{R}_\ast
\end{equation}
and furthermore by a three term geometric expansion also 
\begin{equation}
\norm{\BO^{-1}\cQ}_{\ast\to\ast} \le (1+\norm{\cQ}_{\ast\to\ast})\Big( 1+ \norm{\cC_M \SS}_{\ast\to\ast} + \norm{\cC_M \SS}_{\ast\to\text{hs}} \normsp{\BO^{-1}\cQ}\norm{\cC_M\SS}_{\text{hs}\to\ast}\Big).
\end{equation}
\end{subequations}
Since 
\begin{equation*}
\norm{\cP[R]}_\ast = \frac{\abs{\braket{P,R}}}{\abs{\braket{P,B}}} \norm{B}_\ast \le \frac{\norm{B}}{\abs{\braket{P,B}} N} \sum_a \abs{R_{P^\ast_{a\cdot}a}} 
\le \frac{\norm{B}\norm{R}_\ast}{\abs{\braket{P,B}}  N}\sum_a \norm{P^\ast_{a\cdot}} \le \frac{\norm{P}\norm{B}}{\abs{\braket{P,B}}} \norm{R}_\ast
\end{equation*}
it follows that $\norm{\cP}_{\ast\to\ast}\lesssim 1$ and therefore also $\norm{\cQ}_{\ast\to\ast}\lesssim 1$. Now, since $\norm{R}_{\max}\le\norm{R}_\ast\le\norm{R}$ and according to (73) in \cite{2017arXiv170510661E} also $\max\{\norm{\SS}_{\max\to\norm{\cdot}},\norm{\SS}_{\text{hs}\to\norm{\cdot}}\}\lesssim 1$, the lemma follows together with $\normsp{\BO^{-1}\cQ}\lesssim 1$ from Proposition \ref{prop stability reference}\eqref{1-CMS bound with flatness}.
\end{proof}

\begin{lemma}\label{hs local law}
Fix any $\epsilon,\delta>0$ and an integer $k\ge 0$. Under the assumptions of Theorem \ref{theorem local law}, for the $k$-th derivatives of $M$ and $G$ we have the bound 
\begin{equation}\label{hs local law eq}
\abs{\braket{G^{(k)}(z)-M^{(k)}(z)}}\prec \frac{1}{N\kappa^{k+1}}.
\end{equation}
uniformly in $z\in\DD^\delta$ with $\kappa=\dist(z,\supp\varrho)\ge N^{-2/3+\epsilon}$.
\end{lemma}
\begin{proof}
We will fix $z=x+\ii\eta$ throughout the proof. Let $\chi\colon\R\to\R$ be a smooth cut-off function such that $\chi(x')=1$ for $\kappa'=\dist(x',\supp\varrho)\le \kappa/3$ and $\chi(x')=0$ for $\kappa'\ge 2\kappa/3$ and let $\wt\chi$ be a cut-off function such that $\wt\chi(\eta')=1$ for $\eta'\le 1$ and $\wt\chi(\eta')=0$ for $\eta'\ge 2$. We also assume that the cut-off functions have bounded derivatives in the sense $\norm{\chi'}_\infty\lesssim 1/\kappa, \norm{\chi''}_\infty\lesssim 1/\kappa^2$ and $\norm{\wt\chi'}_\infty\lesssim 1$. We now define $f(x')\defeq (x'-z)^{-k}\chi(x')$ and the almost analytic extension 
\begin{equation*}
f^{\C}(z')=f^{\C}(x'+\ii\eta') \defeq \wt\chi(\eta')\Big[ f(x')+\ii\eta'f'(x') \Big], \qquad \partial_{\overline{z}} f^{\C}(z')=\frac{\ii\eta'}{2}\wt\chi(\eta')f''(x') + \frac{\ii}{2} \wt\chi'(\eta')\Big[ f(x')+\ii\eta' f'(x') \Big].
\end{equation*}
It follows from the Cauchy Theorem and the absence of eigenvalues outside $\set{\chi=1}$ in the sense of Corollary \ref{cor no eigenvalues outside} that with high probability
\begin{equation*}
\braket{G^{(k)}(z)-M^{(k)}(z)} = \frac{2}{\pi}\Re\int_{\R}\int_{\R_+} \partial_{\overline z} f^{\C}(z') \braket{G(z')-M(z')}\diff\eta'\diff x'.
\end{equation*}
Due to the fact that $\wt\chi'=0$ for $\eta'\le1$ the second term in $\partial_{\overline z} f^{\C}$ only gives a contribution of $1/N\kappa^{k+1}$ even by the local law and the $\norm{\cdot}_\infty$ bound for $\partial_{\overline z}f^{\C}$ and we now concentrate on the first term. First, we exclude the integration regime $\eta'\lesssim N^{-1+\gamma}$ in which we cannot use the local law but only the trivial bound $\braket{G-M}\lesssim 1/\eta'$. For the contribution of this regime to \eqref{hs local law eq} we thus have to estimate 
\begin{equation*}
N^{-1+\gamma}\int_{\R} \abs{f''(x')} \diff x' \lesssim \frac{1}{N} \int_{\abs{x-x'}\ge 2\kappa/3}\bigg[\frac{1}{\kappa^2 \abs{x-x'}^{k}}+\frac{1}{\kappa \abs{x-x'}^{k+1}}+\frac{1}{ \abs{x-x'}^{k+2}} \bigg]\diff x' \lesssim \frac{N^{\gamma}}{N\kappa^{k+1}}
\end{equation*}
and we have shown that
\begin{equation*}
\abs{\braket{G^{(k)}(z)-M^{(k)}(z)}} \prec \frac{N^\gamma}{N\kappa^{k+1}} + \int_{\R} \int_{N^{-1+\gamma}}^2 \eta' \bigg[\frac{\chi(x')}{\abs{x'-z}^{k+2}}+\frac{\chi'(x')}{\abs{x'-z}^{k+1}}+\frac{\chi''(x')}{\abs{x'-z}^{k}}\bigg] \abs{\braket{G(z')-M(z')}}\diff \eta'\diff x'.
\end{equation*}
We now use the local law of the form $\abs{\braket{G-M}}\prec 1/N(\kappa+\eta')$ and that in the second and third term the integration regime is only of order $\kappa$ to obtain the final bound of $N^{\gamma}/N\kappa^{k+1}$ for any $\gamma>0$.
\end{proof}

\printbibliography

\end{document}